\documentclass[aos,preprint]{imsart}
\RequirePackage[OT1]{fontenc}
\RequirePackage{amsthm,amsmath,fancyhdr,dsfont,enumerate,float,indentfirst,latexsym,a4,mathrsfs,amssymb,url,tikz,epstopdf}
\RequirePackage[colorlinks,citecolor=blue,urlcolor=blue]{hyperref}
\RequirePackage[numbers]{natbib}
\usepackage[leqno,math]{easyeqn}

\startlocaldefs
\numberwithin{equation}{section}
\theoremstyle{plain}
\newtheorem{theorem}{Theorem}[section]
\newtheorem{lemma}[theorem]{Lemma}
\theoremstyle{remark}
\newtheorem{remark}{Remark}[section]
\newtheorem{corollary}[theorem]{Corollary}
\endlocaldefs

\def\sbt{\circ}
\def\t{\mathsf{T}}

\def\Gh{\widehat{\G}}
\def\Tset{\mathcal{T}}

\newcommand{\R}{\mathbb{R}}
\newcommand{\C}{\mathbb{C}}
\newcommand{\HM}{\mathbb{H}}
\newcommand{\T}{\mathsf{T}}

\newcommand{\X}{\mathbf{X}}

\newcommand{\A}{\mathbf{A}}
\newcommand{\CC}{\mathbf{C}}

\newcommand{\B}{\mathbf{B}}

\newcommand{\EE}{\mathbf{E}}

\newcommand{\G}{\mathbf{G}}

\newcommand{\OO}{\mathbf{O}}

\newcommand{\II}{\mathbf{I}}

\newcommand{\V}{\mathbf{V}}

\newcommand{\CM}{\mathbf{\Sigma}}
\newcommand{\SC}{ \widehat{\CM} }
\newcommand{\SCb}{\CM^{\sbt} }
\newcommand{\PP}{\mathbf{P}}
\newcommand{\PS}{\widehat{\PP}}
\newcommand{\PB}{\PP^{\sbt}}

\newcommand{\W}{\mathbf{W}}

\newcommand{\ee}{\mathbf{e}}
\newcommand{\uu}{\mathbf{u}}

\newcommand{\ub}{\uu^{\sbt}}

\newcommand{\vv}{\mathbf{v}}

\newcommand{\ve}{\varepsilon}

\def\CMt{\widetilde{\CM}}
\def\EEt{\widetilde{\EE}}
\def\EEh{\widehat{\EE}}
\def\EEb{\EE^{\sbt}}
\def\EEhb{\EEh^{\sbt}}
\def\xib{\xi^{\sbt}}
\def\Gammab{\Gamma^{\sbt}}
\def\PPt{\widetilde{\PP}}
\def\rr{\mathtt{r}}
\def\mm{\mathfrak{m}}

\def\gu{\overline{g}}
\def\etau{\overline{\eta}}
\def\deltau{\overline{\delta}}
\def\Yu{\overline{Y}}
\def\gammab{\gamma^{\sbt}}
\def\sigmab{\sigma^{\sbt}}
\def\sigmat{\widetilde{\sigma}}

\def\err{\diamondsuit}
\def\sigmah{\widehat{\sigma}}
\def\uuh{\widehat{\uu}}
\def\Eset{\mathcal{E}}
\def\Rerr{\mathcal{R}}

\DeclareMathOperator{\Tr}{Tr}
\DeclareMathOperator{\ND}{\mathcal{N}}
\DeclareMathOperator{\cond}{\big|}

\DeclareMathOperator{\Rank}{Rank}
\DeclareMathOperator{\KL}{KL}

\DeclareMathOperator{\E}{\mathbb{E}}
\DeclareMathOperator{\Eb}{\E^{\sbt}}
\DeclareMathOperator{\Cov}{Cov}
\DeclareMathOperator{\Var}{Var}
\DeclareMathOperator{\Covb}{\Cov^{\sbt}}
\DeclareMathOperator{\Varb}{\Var^{\sbt}}
\DeclareMathOperator{\Pb}{\mathbb{P}}
\DeclareMathOperator{\Pbb}{\Pb^{\sbt}}

\def\nquad{\hspace{-1cm}}

\DeclareMathOperator{\eqdef}{\,\,\mathrel{\overset{\makebox[0pt]{\mbox{\normalfont\tiny\sffamily def}}}{=}}\,\,}
\newcommand\myeq{\mathrel{\overset{\makebox[0pt]{\mbox{\normalfont\tiny\sffamily d}}}{=}}}

\DeclareMathOperator{\eqdefapp}{\,\,\mathrel{\overset{\makebox[0pt]{\mbox{\normalfont\tiny\sffamily def}}}{\asymp}}\,\,}
\DeclareMathOperator{\tow}{\,\,\mathrel{\overset{\makebox[0pt]{\mbox{\normalfont\tiny\sffamily w}}}{\longrightarrow}}\,\,}

\begin{document}
\setattribute{journal}{name}{}.

\begin{frontmatter}
\title{Bootstrap confidence sets for spectral projectors of sample covariance
}
\runtitle{Bootstrap confidence sets}

\begin{aug}
\author{\fnms{A.} \snm{Naumov}\thanksref{m5, m1, t1}\ead[label=e1]{a.naumov@skoltech.ru}},
\author{\fnms{V.} \snm{Spokoiny}\thanksref{m2, m1,t2}\ead[label=e2]{spokoiny@wias-berlin.de}}
\and
\author{\fnms{V.} \snm{Ulyanov}\thanksref{m3,m4}
\ead[label=e3]{vulyanov@cs.msu.ru}
}

\thankstext{t1}{Supported by RFBR N~16-31-00005 and Russian President Fellowship for young scientists N~4596.2016.1.}
\thankstext{t2}{Supported by the Russian Science Foundation grant (project 14-50-00150).}
\runauthor{A. Naumov et al.}

\affiliation{
	Weierstrass Institute for Applied Analysis and Stochastics\thanksmark{m2};\\ 
	Skolkovo Institute of Science and Technology \thanksmark{m5};\\
	IITP RAS \thanksmark{m1}; \\
	Lomonosov Moscow State University\thanksmark{m3};\\  
	National Research University Higher School of Economics (HSE) \thanksmark{m4}}

\address{A. Naumov,\\
Skolkovo Institute of Science and Technology\\
Skolkovo Innovation Center, \\
Bld. 3, 143026, Moscow, Russia;\\
\printead{e1}}

\address{V. Spokoiny\\
	Weierstrass-Institute \\
	Mohrenstr. 39\\
	10117 Berlin\\
	Germany\\
	\printead{e2}
}

\address{V. Ulyanov\\
Faculty of Computational Mathematics and Cybernetics\\
Moscow State University \\
Moscow, 119991, Russia
\printead{e3}\\}
\end{aug}

\begin{abstract}
Let \( X_{1},\ldots,X_{n} \) be i.i.d. sample in \( \R^{p} \) with zero mean and the covariance 
matrix \( \CM \). 
The problem of recovering the projector onto an eigenspace of \( \CM \) 
from these observations naturally arises in many applications. 
Recent technique from \cite{KoltchLounici2015} helps to study the asymptotic distribution 
of the distance in the Frobenius norm \( \| \PP_{r} - \PS_{r} \|_{2} \) 
between the true projector
\( \PP_{r} \) on the subspace of the \( r \)th eigenvalue and its empirical counterpart 
\( \PS_{r} \) in terms of the effective rank of \( \CM \).
This paper offers a bootstrap procedure for building sharp confidence sets for the true projector 
\( \PP_{r} \) from the given data. 
This procedure does not rely on the asymptotic distribution of 
\( \| \PP_{r} - \PS_{r} \|_{2} \) and its moments. It could be applied for small or moderate sample size \( n \) and large dimension \( p \). The main result states the validity of the proposed procedure for finite samples
with an explicit error bound for the error of bootstrap approximation.
This bound involves some new sharp results on Gaussian comparison and Gaussian anti-concentration in high-dimensional spaces.
Numeric results confirm a good performance of the method in realistic examples.
\end{abstract}

\begin{keyword}[class=MSC]
\kwd[Primary ]{60K35}
\kwd{60K35}
\kwd[; secondary ]{60K35}
\end{keyword}

\begin{keyword}
\kwd{sample covariance matrices, bootstrap, spectral projectors, Gaussian comparison and anti-concentration inequalities, effective rank}
\kwd{\LaTeXe}
\end{keyword}

\end{frontmatter}

\section{Introduction}

Let  \( X, X_{1}, \ldots , X_{n} \) be independent identically distributed (i.i.d.) random vectors taking values in \( \R^{p} \) with mean zero and \( \E\|X\|^{2} < \infty \). 
Denote by \( \CM \) its \( p \times p \) symmetric covariance matrix defined as
\begin{EQA}
	\CM 
	& \eqdef &
	\E (X X^{\T}).
	\label{CMdefEXX}
\end{EQA} 
We also consider the sample covariance matrix \( \SC \) of the observations \( X_{1}, \ldots , X_{n} \) 
defined as the average of \( X_{j} X_{j}^{\T} \):
\begin{EQA}
	\SC 
	&\eqdef & 
	\frac{1}{n} \sum_{j=1}^{n} X_{j}  X_{j}^{\T} = \frac{1}{n} \X \X^{\T}, 
\end{EQA}
where \( \X \eqdef  [X_{1}, \ldots , X_{n}] \in \R^{p \times n} \). 

In statistical applications, the true covariance matrix \( \CM \) is typically unknown and 
one often uses the sample covariance matrix \( \SC \) as its estimator. 
The accuracy \( \| \SC - \CM \| \) of estimation of \( \CM \) by \( \SC \),
in particular, for \( p \) much larger than \( n \),
has been actively studied in the literature.
We refer to \cite{Tropp2012} for an overview of the recent results based on the matrix Bernstein 
inequality; see also 
\cite{Vershynin2012}  
and 
\cite{vanHandel2015}. 
A bound in term of the effective rank \( \rr(\CM) \eqdef \Tr(\CM)/\| \CM \| \) can be found in
\cite{KoltchLounici2015b}. 
This or similar bounds on the spectral norm \( \| \SC - \CM \| \) can be effectively applied 
to relate the eigenvalues of \( \CM \) and of \( \SC \) under the spectral gap condition.
This paper focuses on a slightly different problem of recovering the spectral projectors 
on the eigen-subspaces of \( \CM \) for few significantly positive eigenvalues.
Such tasks naturally arise in 
many dimensionality reduction techniques for large \( p \).
In particular, the famous principal component analysis (PCA) projects the vector \( X \) onto the subspace 
spanned by the eigenvectors for the first principal eigenvalues.
A significant error in recovering these eigenvectors would lead to a substantial loss of information 
contained in the data by PCA projection. 
The popular Sliced Inverse Regression (SIR) method under the assumption of elliptically contoured  distributions for high dimensional or functional data leads back to recovering the eigen-subspace 
from a finite sample; see e.g. \cite{li2010} and references therein. 
The use of dimension reduction methods in deep networking architecture is discussed 
in \cite{Goodfellow-et-al-2016-Book} among others.
We also mention the use of dimension reduction technique in numerical integration with applications to finance and 
insurance; see e.g. 
\cite{holtz2010sparse}.
Justification of the assumption of low effective dimension in financial problems can be found in
\cite{WangSloan2005} among many others.

Surprisingly,
the problem of recovering the spectral projectors (eigenvectors or eigen-subspaces) of \( \CM \) 
from the sample \( X_{1},\ldots,X_{n} \) for significantly positive spectral values is much less studied than the problem of recovering the covariance matrix \( \CM \). 
Recently \cite{KoltchLounici2015} established sharp non-asymptotic bounds on the
Frobenius distance \( \| \PP_{r} - \PS_{r} \|_{2} \) between 
the spectral projectors \( \PP_{r} \) and its empirical counterparts \( \PS_{r} \)
for the \( r \)th eigenvalue, as well as its asymptotic behaviour for large samples. 
This enables to build some asymptotic confidence sets for the target projector \( \PP_{r} \) 
as a proper elliptic vicinity of \( \PS_{r} \). 
However, it is well known that such asymptotic results apply only for really large samples 
due to a slow convergence of the normalized U-statistics to the limiting normal law. 

The aim of this paper is to develop and validate a bootstrap procedure for building 
a confidence set for \( \PP_{r} \) which is applied for small or moderate samples and for large 
dimension \( p \).
Bootstrap method is nowadays one of the most popular way for measuring the significance of 
a test or for building a confidence sets. 
The existing theory based on the high order expansions of the related statistics states 
the bootstrap validity for various parametric methods. 
However, an extension to a non-classical situation with a limited sample size and/or high parameter 
dimension meets serious problems.
We refer to series of works \cite{Chernozhukov2013}, \cite{Chernozhukov2014} which validate a bootstrap procedure 
for a test based on the maximum of huge number of statistics. 
In particular, the authors emphasised a close relation between bootstrap validity results and 
the so called ``anticoncentration'' bounds on the Levy measure for rectangle sets.
The paper \cite{spokoiny2015} studies applicability of the likelihood based statistics for finite samples 
and large parameter dimension under possible model misspecification. 
The important step in the proof of bootstrap validity was again based on a kind of 
``anticoncentration bound'' but now for spherical sets. 

This paper makes a further step in understanding the range of applicability of a weighted bootstrap method
in constructing a finite sample confidence set for a spectral projector.
A proof of bootstrap validity in this setup is a challenging task. 
The spectral projector is a nonlinear and non-regular function of the 
covariance matrix, which itself is a quadratic function of the underlying multivariate
distribution.  
In situations with high-dimensional space and small or moderate sample size the classical
asymptotic methods of bootstrap validation do not apply. 
It appears that even in a Gaussian case the proof of bootstrap consistency requires to develop new  
probabilistic tools for establishing some sharp anticoncentration bounds for Gaussian measures in high-dimensional or even infinite dimensional Hilbert spaces. 
One more technical difficulty is that the bootstrap measure is random 
and depends upon the sample \( \X \). 
The same applies to all corresponding probabilities, that is, bootstrap quantiles are random
and data dependent. 
The main contributions of this paper are:
\begin{itemize}
	\item
	we offer a new bootstrap procedure for recovering the spectral projector on 
	a low dimensional eigen-subspace;
	\item
	the validity of this procedure is proved under rather general and mild conditions.
	We present a non-asymptotic upper bound 
	for the accuracy of bootstrap approximation. 
	The error term is dimension free and the bound applies even for the dimension \( p \)
	which is exponential in the sample size. 
	The result also applies for small or moderate samples;
	\item 
	a numerical study illustrates a very good performance of the proposed procedure
	in realistic setups;
	\item 
	we establish new sharp results on Gaussian 
	comparison and Gaussian anti-concentration which are heavily used for proving 
	the validity of the bootstrap procedure but they are
	probably of independent interest; see 
	Lemmas~\ref{l: explicit gaussian comparison} and \ref{band of GE} below. 
\end{itemize}

The paper is organized as follows.
The next section contains the description of the bootstrap procedure and the main results
about its validity.
Numerical results of Section~\ref{SnumresultsPP} illustrate the performance of the procedure for finite samples.
Main proofs are collected in Section~\ref{SproofsPP}.    The results on Gaussian 
comparison and Gaussian anti-concentration see in Section~\ref{SgaussianPP}.  Appendix~\ref{SappendPP} gathers some auxilary statements and existing results. 

Throughout the paper we will use the following notations. 
\( \R \) (resp. \( \C \)) denotes the set of all real (resp. complex) numbers.
We assume that all random variables are defined on common probability space \( (\Omega, \mathfrak{F}, \Pb) \) and let \( \E \) be the mathematical expectation with respect to \( \Pb \).  
\(\mathfrak{B}(\R^p)\) means the Borel \(\sigma\)-algebra in \( \R^{p} \). 
For a vector \( \uu \), by \( \| \uu \| \) we denote 
its natural Euclidean norm. 
For a matrix \( \A \in \R^{N\times N} \), we denote its rank and trace by \( \Rank \A \) and \( \Tr \A \) resp. Let \( \|\A \| \eqdef \sup_{\|x\|=1} \|\A x\| \). For a symmetric operator \(\A\) we define the Schatten \(p\)-norm by \( \|\A \|_p \eqdef \bigl(\sum_{k=1}^\infty |\lambda_k(\A)|^p\bigr)^{1/p} \), where \(\lambda_1(\A), \lambda_2(\A), \dots\) are the eigenvalues of \(\A\). In particular, \( \|\A\|_{2}\) is the Hilbert-Schmidt (Frobenius) norm of \(\A\). For symmetric positive-definite matrix \(\A\) we define its effective rank by \(\rr(\A) \eqdef {\Tr \A}/{\|\A\|}\). We write \( a \lesssim b \) (\( a \gtrsim b \)) if there exists some absolute constant \( C \) such that \( a \le C b \) (\( a \geq C b \) resp.). 
Similarly, \(a \asymp b\) means that there exist \(c, C\) such that \(c \, a \le b \le C \, a\). 
For r.v. \( X \) and \( Y \) we write \( X \myeq Y \) if they are equally distributed.

\section{Procedure and main results}
\label{SprocmainPP}
This section presents the bootstrap procedure for building a confidence set for the true projector 
\( \PP_{r} \) and states the result about its validity.

\subsection{Setup and problem}
Let \( \sigma_{1} \geq \sigma_{2} \geq \ldots \geq \sigma_{p} \) be the eigenvalues of \( \CM \) and 
\( \uu_{j}, j = 1, \ldots , p\), be the corresponding orthonormal eigenvectors. 
Matrix \( \CM \) has the following spectral decomposition 
\begin{EQA}
	\CM &=& \sum_{j = 1}^{p} \sigma_{j} \uu_{j} \uu_{j}^{\T}.	\label{eq: spectral decomposition}
\end{EQA}
Let \( \mu_{1} > \mu_{2} > \ldots > \mu_{q} > 0\) with some \( 1 \le q \le p \), be strictly distinct eigenvalues of \( \CM \) and \( \PP_{r}, r = 1, \ldots, q\), be the corresponding spectral projectors (orthogonal projectors in \( \R^{p} \)). 
Denote \( m_{r} \eqdef  \Rank(\PP_{r}) \). 
We may rewrite~(\refeq{eq: spectral decomposition}) in terms of distinct eigenvalues and corresponding spectral projectors, namely
\begin{EQA}
	\CM &=& \sum_{r = 1}^{q} \mu_{r} \PP_{r}. 	\label{eq: spectral decomposition 2}
\end{EQA}
Denote by \( \Delta_{r} \eqdef  \{j\colon \sigma_{j} = \mu_{r}\} \). 
Then \( |\Delta_{r}| = m_{r} \). 
Define \( g_{r} \eqdef  \mu_{r} - \mu_{r+1} >0\) for \( r \geq 1 \). 
Let \( \gu_{r} \eqdef  \min(g_{r-1}, g_{r}) \) for \( r \geq 2 \) and \( \gu_{1} \eqdef  g_{1} \). 
The quantity \( \gu_{r} \) is the \( r \)-th spectral gap of the eigenvalue \( \mu_{r} \). 

Consider now the sample covariance matrix \( \SC \). 
Similarly to~(\refeq{eq: spectral decomposition}), it can be represented as
\begin{EQA}
	\SC &=& \sum_{j = 1}^{p} \sigmah_{j} \uuh_{j} \uuh_{j}^{\T},
\end{EQA}
where \( \sigmah_{1} \geq \sigmah_{2} \geq \ldots \geq \sigmah_p, \uuh_{1}, \ldots , \uuh_p \) are the eigenvalues and the corresponding eigenvectors of \( \SC \). 
Following~\cite{KoltchLounici2015} we may define clusters of eigenvalues \( \sigmah_{j}, j \in \Delta_{r} \). 
Let \( \EEh \eqdef  \SC - \CM \). 
One may show that
\begin{EQA}
	\inf_{j \notin \Delta_{r}}|\sigmah_{j} - \mu_{r}| 
	&\geq &
	\gu_{r} - \|\EEh\|,
	\\
	\sup_{j \in \Delta_{r}}|\sigmah_{j} - \mu_{r}| 
	&\leq &
	\|\EEh\|.
\end{EQA}
Assume that \( \|\EEh\| \le \gu_{r}/2 \). 
Then  all \( \sigmah_{j}, j \in \Delta_{r} \) may be covered by an interval 
\begin{EQA}
	(\mu_{r} - \|\EEh\|, \mu_{r} + \|\EEh\|) 
	&\subset &
	\left(\mu_{r} - \frac{\gu_{r}}{2}, \mu_{r} + \frac{\gu_{r}}{2} \right).
\end{EQA}
The rest of the eigenvalues of \( \SC \) are outside of the interval
\begin{EQA}
	\Bigl( \mu_{r} - (\gu_{r} - \|\EEh\|), \mu_{r} + (\gu_{r} - \|\EEh\|) \Bigr) 
	&\supset &
	\left [\mu_{r} - \frac{\gu_{r}}{2}, \mu_{r} + \frac{\gu_{r}}{2}\right].
\end{EQA}
Let \( \|\EEh\| < \frac14 \min_{1 \le s \le r} \gu_{s} = :\deltau_{r} \). 
The set \( \{\sigmah_{j}, j \in \cup_{s=1}^{r} \Delta_{s} \} \) consists of \( r \) clusters, the diameter of each cluster being strictly smaller than \( 2 \deltau_{r} \) and the distance between any two clusters being larger than \( 2 \deltau_{r} \). 
We denote by \( \PS_{r} \) the projector on subspace spanned by the direct sum of \( \uuh_{j}, j \in \Delta_{r} \).  
The asymptotic behavior of 
\( \|\PS_{r} - \PP_{r}\|_{2}^{2} \) can be used for building sharp asymptotic confidence sets for the unknown
projector \( \PP_{r} \). 
It follows from~\cite{KoltchLounici2015}[Theorem 5] that 
\begin{EQA}
	\frac{\|\PS_{r} - \PP_{r}\|_{2}^{2} - \E \|\PS_{r} - \PP_{r}\|_{2}^{2}}
	{\Var^{1/2}(\|\PS_{r} - \PP_{r}\|_{2}^{2})} 
	& \tow &
	\ND(0,1), \label{PRshPSr2twN}
\end{EQA}
that is, after centering and normalization, the error \( \|\PS_{r} - \PP_{r}\|_{2}^{2} \) 
is asymptotically standard normal. 
This allows to build an asymptotic elliptic confidence set for \( \PP_{r} \) in the form
\begin{EQA}
	\biggl\{ \PP_{r} \colon  
	\frac{\|\PS_{r} - \PP_{r} \|_{2}^{2} - \E \|\PS_{r} - \PP_{r}\|_{2}^{2}}
	{\Var^{1/2}(\|\PS_{r} - \PP_{r}\|_{2}^{2})}
	&\leq &
	z_{\alpha}
	\biggr\}
\end{EQA}
where \( z_{\alpha} \) is a proper quantile of the standard normal law.
However, there are at least two drawbacks of this approach. First, the weak convergence in~(\refeq{PRshPSr2twN}) is very slow and it requires astronomic sample size to achieve a reasonable quality
of approximation.
Second, to apply this construction in practice we need to know or to estimate the values 
\( \E \|\PS_{r} - \PP_{r}\|_{2}^{2} \) and 
\( \Var(\|\PS_{r} - \PP_{r}\|_{2}^{2}) \) which depends on the unknown covariance operator 
\( \CM \).  
\cite{KoltchLounici2015} offered a procedure which splits the sample into three subsamples,
one for estimating the expectation and another one for estimating the variance of 
\( \|\PS_{r} - \PP_{r}\|_{2}^{2} \). 
The remaining data can be used for building the confidence set. 
The present paper proposes another procedure which 
\begin{itemize}
	\item
	does not rely on the asymptotic distribution
	of the error \( \|\PS_{r} - \PP_{r}\|_{2}^{2} \),
	\item
	does not require to know the moments
	of \( \|\PS_{r} - \PP_{r}\|_{2}^{2} \),  
	\item
	does not involve any data splitting,
	\item
	provides an explicit error bound for the bootstrap approximation. 
\end{itemize}

The procedure is based on the resampling idea which allows to estimate directly the quantiles
\begin{EQA}
	\gamma_{\alpha} 
	&\eqdef & 
	\inf \left \{
	\gamma > 0 \colon \Pb\left( n\|\PS_{r} - \PP_{r}\|_{2}^{2} > \gamma \right) \leq \alpha  
	\right\} \label{eq: quantile}
\end{EQA}
without estimating the covariance matrix \( \CM \). 
The introduced bootstrap procedure is described in the next section. 

\subsection{Bootstrap procedure} 
We introduce the following weighted version of \( \SC \), namely
\begin{EQA}
	\SCb &\eqdef & \frac{1}{n} \sum_{i=1}^{n} w_{i} X_{i} X_{i}^{\T},
\end{EQA}
where \( w_{1}, \ldots , w_{n} \) are i.i.d. random variables, 
independent of \( \X = (X_{1}, \ldots , X_{n}) \), with \( \E w_{1} = 1\), \( \Var w_{1} = 1 \). 
A typical example used in this paper is to apply i.i.d. Gaussian weights 
\( w_{i} \sim \ND(1,1) \).
We denote by \( \Pbb(\cdot) \eqdef \Pb(\cdot \cond \X) \) and \( \Eb \) corresponding conditional probability and expectation. 
It is straightforward to check that
\begin{EQA}
\label{eq: math exp of sample cov}
	\Eb\SCb &=& \SC. 
\end{EQA}
In what follows we will often refer to "\(\X\) - world" and "bootstrap world". In the \( \X \) - world the sample \( \X \) is random opposite to the bootstrap world, where \(\X\) is fixed, but \( w_{1}, \ldots, w_n \) are random.  
Then, equation~(\refeq{eq: math exp of sample cov}) implies that in the bootstrap world we know precisely the expectation of \( \SCb \) opposite to the \( \X \) - world, where \( \CM \) is unknown.
Similarly to~(\refeq{eq: spectral decomposition}) we may write
\begin{EQA}
	\SCb &=& \sum_{j = 1}^{p} \sigmab_{j} \ub_{j} {\ub_{j}}^{\T}.
\end{EQA}
Let us denote by \( \PB_{r} \) a projector on the subspace spanned by the direct sum of \( \ub_{j}, j \in \Delta_{r} \). For a given \( \alpha \) we define the quantile \( \gammab_{\alpha} \) as
\begin{EQA}
\label{def gamma null}
	\gammab_{\alpha}
	&\eqdef & 
	\min \left\{ 
		\gamma>0 \colon \Pbb\left(n\| \PB_{r} - \PS_{r}\|_{2}^{2} > \gamma\right) \le \alpha 
	\right\}.
\end{EQA}
Note that this value \( \gammab_{\alpha} \) is defined w.r.t. the bootstrap measure, therefore, it depends on the data \( \X \). 
This bootstrap critical value \( \gammab_{\alpha} \) is applied in the \( \X \) - world to build the confidence set
\begin{EQA}
\label{condsetPP}
	\Eset(\alpha) 
	&\eqdef &
	\bigl\{ \PP \colon n \| \PP - \PS_{r} \|_{2}^{2} \le \gammab_{\alpha} \bigr\}. 
\end{EQA}
The main result given   in the next section justifies this construction and evaluate the coverage probability of the true projector \( \PP_{r} \) by this set. It states that
\begin{EQA}
	\Pb( \PP_{r} \not\in \Eset(\alpha) \bigr)
	& = &
	\Pb( n \| \PP_{r} - \PS_{r} \|_{2}^{2} > \gammab_{\alpha} \bigr)
	\approx
	\alpha .
\end{EQA}

\subsection{Main results. Bootstrap validity}
To formulate the main result of this paper we introduce additional notation. Define the following block-matrix
\begin{EQA}
	\label{eq:cov-matrix-gamma}
	\Gamma_{r} 
	&\eqdef &
	\begin{pmatrix}
		\Gamma_{r 1} & \OO & \ldots & \OO \\
		\OO & \Gamma_{r 2} & \OO \ldots & \OO\\
		\ldots \\
		\OO & \ldots & \OO & \Gamma_{r q}
	\end{pmatrix},
\end{EQA}
where \( \Gamma_{rs}, s \neq r \)
are diagonal matrices of order \( m_{r} m_{s}\times m_{r} m_{s} \) with values \( {2\mu_{r} \mu_{s}}/{(\mu_{r} - \mu_{s})^{2}} \) on the main diagonal. 
Let \( \lambda_{1}(\Gamma_{r}) \geq \lambda_{2}(\Gamma_{r}) \geq \ldots  \) be the eigenvalues of 
\( \Gamma_{r} \).

The available bounds on the distance between the covariance matrix and 
its empirical counterpart claim that the eigenvalues of \( \CM \) can be recovered 
with accuracy \(O(1/\sqrt{n})\). 
Therefore, the part of the spectrum of \( \CM \) below a  threshold of order \(O(1/\sqrt{n})\)  
cannot be estimated. 
The same applies to the matrix \( \Gamma_{r} \).
Introduce the corresponding value \( \mm \): 
\begin{EQA}
\label{strangemPP}
	\lambda_{\mm}(\Gamma_{r}) 
	&\geq &
	\Tr \Gamma_{r} \left(\sqrt{\frac{\log n}{n}} + \sqrt{\frac{\log p}{n}} \right) \geq \lambda_{\mm+1}(\Gamma_{r}).
\end{EQA}
Denote by \( \Pi_{\mm} \) a projector on the subspace spanned by the eigenvectors of \(\Gamma_{r}\) corresponding to its largest \( \mm \) eigenvalues. The main result is the following theorem.

\begin{theorem}
	\label{th: main}
	Let observations \( X, X_{1}, \ldots , X_{n} \) be i.i.d. Gaussian random vectors in 
	\( \R^{p} \) with \( \E X = 0 \) and \( \E X X^{\T} = \CM \).
	Let \( \gammab_{\alpha} \) be defined 
	by (\refeq{def gamma null}) for any  \( \alpha: 0 < \alpha < 1 \), with i.i.d. Gaussian random weights \( w_{i} \sim \ND(1,1) \)
	for \( i=1,\ldots,n \). 
	Then the following bound is fulfilled
	\begin{EQA}
		\left|\alpha - \Pb\left(n\|\PS_{r} - \PP_{r}\|_{2}^{2} > \gammab_{\alpha} \right) \right| 
		&\lesssim &
		\err,
	\label{mainbootapp}
	\end{EQA}
	where
	\begin{EQA}
		\err &\eqdef& \frac{ \mm\, \Tr \Gamma_{r} }{\sqrt{\lambda_{1}(\Gamma_{r}) \lambda_{2}(\Gamma_{r})}} \left(\sqrt{\frac{\log n}{n}} + \sqrt{\frac{\log p}{n}} \right) +  \frac{\Tr (\II - \Pi_\mm)\Gamma_{r}}{\sqrt{\lambda_{1}(\Gamma_{r}) \lambda_{2}(\Gamma_{r})}} 
		\\
		\label{def: err}
		&&
		+ \, \frac{m_{r} \Tr^3\CM}{\gu_{r}^{3} \sqrt{\lambda_{1}(\Gamma_{r}) \lambda_{2}(\Gamma_{r})} } \left(\sqrt{\frac{\log ^{3} n}{n}} 
		+ \sqrt{\frac{\log^{3} p}{n}} \right) 
	\end{EQA}
and \(\mm\)  is defined by (\refeq{strangemPP}).
\end{theorem}

\begin{remark}
	The result (\refeq{mainbootapp}) implicitly assumes that the error term \( \err \) is small. 
	If \( \err \geq 1 \) then (\refeq{mainbootapp}) is meaningless. 
	In particular, this implies that 
\begin{EQA}
	 p &\lesssim & 
	 e^{n^{1/3}} . 
\label{plen13b}
\end{EQA}
\end{remark}
\begin{remark}
	The error term \( \err \) can be described in terms of \(\CM \). 
	It is easy to check that for all \( r \) 
	\begin{EQA}
	\label{eq: A_r estimate}
		\Tr \Gamma_{r} \lesssim m_{r} \frac{\mu_{r} \Tr \CM}{\gu_{r}^{2}} 
		&\le &
		m_{r} \frac{\|\CM\|^{2} \rr(\CM)}{\gu_{r}^{2}}.
	\end{EQA}
	Let us consider, for example, the case \( r = 2\) and \(m_1 = m_2 = 1\). Introduce a function \( f(x) = {2 x \mu_2}/{(x - \mu_2)^2}\) at the points \( x = \mu_s, s \neq 2\). It is straightforward to check that the maximum of \(f(x)\) is achieved at \(x = \mu_1\) or \(\mu_3\). Moreover,   assume that the largest values of \(f(x)\) are \(f(\mu_1)\) and \( f(\mu_3)\). Then we may estimate \( \err \) as follows:
	\begin{EQA}
		\err &\lesssim& \frac{ \mm \, \Tr \CM }{\overline g_2} \sqrt{\frac{\mu_1}{\mu_3}} \left(\sqrt{\frac{\log n}{n}} + \sqrt{\frac{\log p}{n}} \right) +  \sqrt{\frac{\mu_1}{\mu_3}} \frac{ \Tr (\II - \Pi_{\mm}) \CM }{\overline g_2} 
		\\
		&&
		+ \, \frac{ \Tr^3\CM}{\gu_{2}^{2} \, \mu_2} \sqrt{\frac{\mu_1}{\mu_3}} \left(\sqrt{\frac{\log ^{3} n}{n}} 
		+ \sqrt{\frac{\log^{3} p}{n}} \right).
	\end{EQA}
\end{remark}

Although an analytic expression for the value \( \gammab_{\alpha} \) is not available,
one can evaluate it from numerical simulations 
by generating a large number of independent samples \( \{w_{1}, \ldots , w_{n}\} \) 
and computing from them the empirical distribution function of \( n \|\PB_{r} - \PS_{r}\|_{2}^{2} \).
Theorem~\ref{th: main} validates the proposed construction of the confidence set (\refeq{condsetPP}),
that is, it justifies the use of this value \( \gammab_{\alpha} \) in place of 
\( \gamma_{\alpha} \) defined in~(\refeq{eq: quantile}) provided that the error \( \err\) is 
sufficiently small.

\section{Numerical results}
\label{SnumresultsPP}

This section illustrates the performance of the bootstrap procedure by means of few artificial examples. 
Namely, we check how well is the bootstrap approximation of the true quantiles. 
We use QQ-plots to compare the distributions of 
\( n \|\PB_{1} - \PS_{1}\|_{2}^{2} \) and \(n\|\PS_{1} - \PP_{1}\|_{2}^{2}\). 

First we describe our setup.
Let \(n\) be a sample size. We consider the different values of \(n\), namely \(n = 100, 300, 500, 1000, 2000, 3000\). 
Let \(X_{1}, \ldots , X_{n}\) have the normal distribution in \(\R^{p}\), with zero mean and covariance matrix \(\CM\). 
The value of \(p\) and the choice of \(\CM\) will be described below. 
The distribution of \(n\|\PS_{1} - \PP_{1}\|_{2}^{2}\) is evaluated by using 3000 Monte-Carlo samples
from the normal distribution with zero mean and covariance \(\CM\). 
The bootstrap distribution for a given realization \( X \) is evaluated by 3000 Monte-Carlo samples of bootstrap weights \( \{w_{1}, \ldots , w_{n}\} \). 
Since this distribution is random and depends on \( X \), we finally use the median from 50 realizations
of \( X \) for each quantile.

In the first example we consider the  following parameters: 
\begin{itemize}
	\item \quad \(p = 500\),
	\item \quad \(\mu_{1} = 36, \mu_{2} = 30, \mu_{3} = 25, \mu_{4} = 19\) and all other eigenvalues \(\mu_{s}, s = 5, \ldots , 500\) are uniformly distributed in \([1,5]\).
\end{itemize}
Here we get \(\gu_{1} = 6\) and \(\rr(\CM) = 51.79\). Figure~\ref{ris:qqp500} shows the corresponding QQ-plots for the empirical distribution of \( n\|\PS_{1} - \PP_{1}\|_{2}^{2} \) against its bootstrap counterpart. Table~\ref{tab 1} shows the coverage probabilities of the quantiles estimated using the bootstrap. 
\begin{figure}[!htb] 
	\includegraphics[width=1\textwidth,height=0.4\textheight]{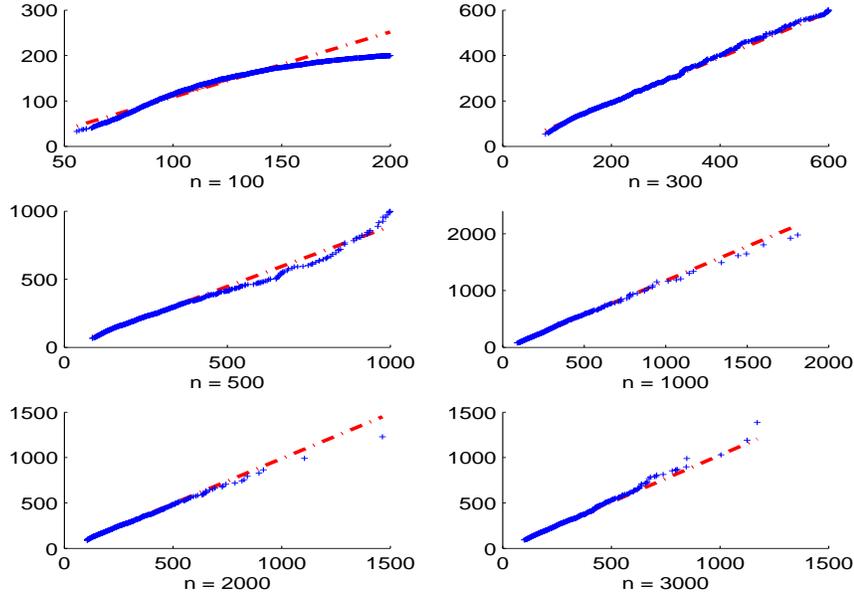}
	\caption[]{QQ-plot of the bootstrap procedure. Here \(p = 500\) and \(\CM\) has the following properties: \(\rr(\CM) = 51.79,\, \mu_{1} = 36, \mu_{2} = 30, \mu_{3} = 25, \mu_{4} = 29\) and all other eigenvalues are uniformly distributed on \([1,5]\).}
	\label{ris:qqp500}
\end{figure}
\begin{table}[!htb]
	\caption{Coverage probabilities. For each \(n\) the first line corresponds to the median value of the coverage probability and the second line corresponds to the interquartile range. Here \(p = 500\) and \(\CM\) has the following properties: \(\rr(\CM) = 51.79,\, \mu_{1} = 36, \mu_{2} = 30, \mu_{3} = 25, \mu_{4} = 29\) and all other eigenvalues are uniformly distributed on \([1,5]\).}
	\label{tab 1}
	\begin{tabular}{crrrrrc}
		\hline
		 \multicolumn{7}{c}{\qquad\qquad Confidence levels}\\
		 \cline{2-7}
		\( n \) & \multicolumn{1}{c}{0.99} & \multicolumn{1}{c}{0.95} & \multicolumn{1}{c}{0.90} & \multicolumn{1}{c}{0.85} &
		\multicolumn{1}{c}{0.80} &
		\multicolumn{1}{c}{0.75}\\
		\hline
		100	&   0.997 &	0.977 &	0.947 &	0.917	& 0.882	& 0.850\\
		&   0.004 &	0.025 &	0.052 &	0.073	& 0.089	& 0.101\\
		300	&   0.992 &	0.965 &	0.933 &	0.883	& 0.834	& 0.779\\    
		&   0.054 &	0.139 &	0.223 &	0.264	& 0.318	& 0.358\\
		500	&   0.992 &	0.965 &	0.933 &	0.883	& 0.834	& 0.779\\
		&   0.026 &	0.093 &	0.163 &	0.208	& 0.233	& 0.262\\
		1000	&   0.992 &	0.965 &	0.933 &	0.883	& 0.834	& 0.779\\
		&   0.021 &	0.063 &	0.114 &	0.149	& 0.161	& 0.177\\
		2000	&   0.988 &	0.945 &	0.883 &	0.836	& 0.789	& 0.731\\
		&   0.021 &	0.060 &	0.086 &	0.104	& 0.125	& 0.141\\
		3000	&   0.994 &	0.951 &	0.900 &	0.857	& 0.808	& 0.751\\   
		&   0.016 &	0.054 &	0.072 &	0.085	& 0.093	& 0.103
		\\ \hline
	\end{tabular}
\end{table}

The second example parameters: 
\begin{itemize}
	\item \quad \(p = 100\),
	\item \quad \(\mu_6, \ldots, \mu_{100}\) are distributed according to Marchenko-Pastur's density with the support on \([0.71, 1.34]\), see~\cite{MarchPastur1967},
	\item \quad all other eigenvalues are \(\mu_{1} = 25.698, \mu_{2} = 15.7688, \mu_{3} = 10.0907, \mu_{4} = 5.9214, \mu_{5} = 3.4321\).
\end{itemize} 
Here \(\gu_{1} = 9,93\) and \(\rr(\CM) = 6.12\). QQ plots are presented on Figure~\ref{ris:qqp100} and the coverage probabilities are collected in Table~\ref{tab 2}.
\begin{figure}[!htb] 
	\includegraphics[width=1\textwidth,height=0.4\textheight]{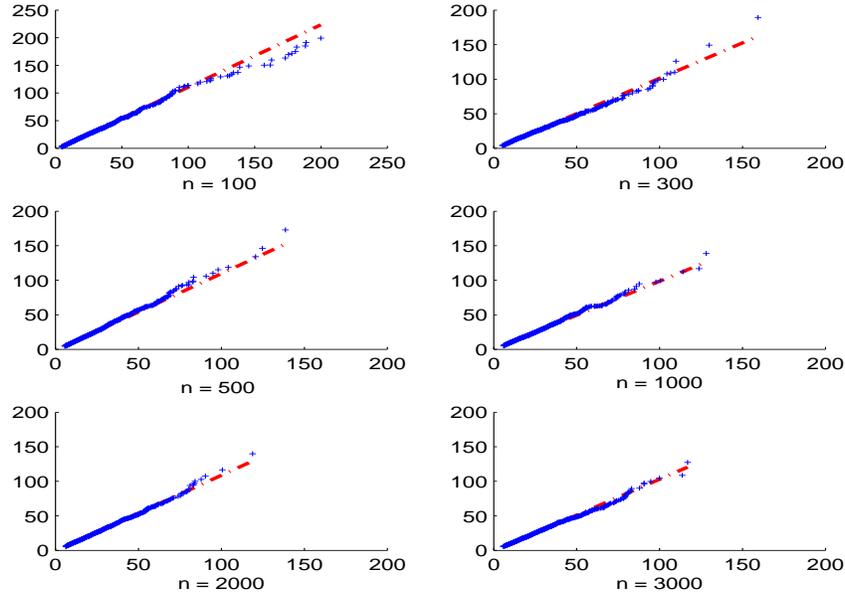}
	\caption[]{QQ-plot of the bootstrap procedure. Here \(p = 100\) and \(\CM\) has the following properties: \(\rr(\CM) = 6.12\) and \(\mu_6, \ldots, \mu_{100}\) are distributed according to Marchenko-Pastur's density with the support on \([0.71, 1.34]\). All other eigenvalues are \(\mu_{1} = 25.698, \mu_{2} = 15.7688, \mu_{3} = 10.0907, \mu_{4} = 5.9214, \mu_{5} = 3.4321\).}
	\label{ris:qqp100}
\end{figure}
\begin{table}[!htb]
	\caption{Coverage probabilities. For each \(n\) the first line corresponds to the median value of the coverage probability and the second line corresponds to the interquartile range. Here \(p = 100\) and \(\CM\) has the following properties: \(\rr(\CM) = 6.12\) and \(\mu_6, \ldots, \mu_{100}\) are distributed according to Marchenko-Pastur's density with the support on \([0.71, 1.34]\). All other eigenvalues are \(\mu_{1} = 25.698, \mu_{2} = 15.7688, \mu_{3} = 10.0907, \mu_{4} = 5.9214, \mu_{5} = 3.4321\).}
	\label{tab 2}
	\begin{tabular}{crrrrrc}
		\hline
		\multicolumn{7}{c}{\qquad\qquad Confidence levels}\\
		\cline{2-7}
		\( n \) & \multicolumn{1}{c}{0.99} & \multicolumn{1}{c}{0.95} & \multicolumn{1}{c}{0.90} & \multicolumn{1}{c}{0.85} &
		\multicolumn{1}{c}{0.80} &
		\multicolumn{1}{c}{0.75}\\
		\hline
		100	& 0.992	&0.961	&0.918	&0.876	&0.825	&0.768\\			
		& 0.027	&0.091	&0.146	&0.197	&0.231	&0.257\\
		300	& 0.988	&0.942	&0.886	&0.832	&0.784	&0.735\\       
		& 0.020	&0.062	&0.094	&0.118	&0.139	&0.153\\
		500	& 0.995	&0.966	&0.925	&0.876	&0.822	&0.771\\            
		& 0.013	&0.035	&0.072	&0.104	&0.120	&0.122\\
		1000      & 0.989	&0.957	&0.906	&0.848	&0.795	&0.743\\            
		& 0.012	&0.038	&0.062	&0.086	&0.093	&0.098\\
		2000	& 0.993	&0.958	&0.913	&0.869	&0.819	&0.775\\
		& 0.011	&0.028	&0.053	&0.065	&0.076	&0.083\\
		3000     & 0.988	&0.952	&0.902	&0.853	&0.803	&0.752\\            
		& 0.006	&0.021	&0.047	&0.053	&0.062	&0.070
		\\ \hline
	\end{tabular}
\end{table}

The third example has the same setup as the previous one except \( \mu_{1} = \mu_{2} = 25.698 \). 
In that case \( \PP_{1} = \uu_{1}\uu_{1}^{\T} + \uu_{2}\uu_{2}^{\T} \). Here \(\gu_{1} = 9,93\)  and \(\rr(\CM) = 6.51\). The result is on Figure~\ref{ris:qq100r12} and Table~\ref{tab 3}.
\begin{figure}[!htb] 
	\includegraphics[width=1\textwidth,height=0.4\textheight]{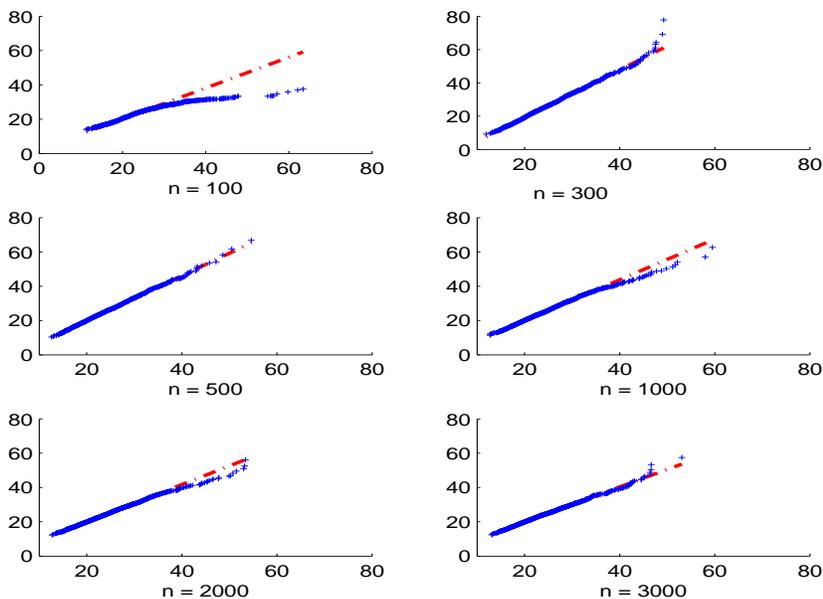}
	\caption[]{QQ-plot of the bootstrap procedure. Here \(p = 100\) and \(\CM\) has the following properties: \(\rr(\CM) = 6.51\) and \(\mu_6, \ldots, \mu_{100}\) are distributed according to Marchenko-Pastur's density with the support on \([0.71, 1.34]\). All other eigenvalues are \(\mu_{1} = \mu_{2} = 25.698, \mu_{3} = 10.0907, \mu_{4} = 5.9214, \mu_{5} = 3.4321\).}
	\label{ris:qq100r12}
\end{figure}
\begin{table}[!htb]
	\caption{Coverage probabilities. For each \(n\) the first line corresponds to the median value of the coverage probability and the second line corresponds to the interquartile range. Here \(p = 100\) and \(\CM\) has the following properties: \(\rr(\CM) = 6.51\) and \(\mu_6, \ldots, \mu_{100}\) are distributed according to Marchenko-Pastur's density with the support on \([0.71, 1.34]\). All other eigenvalues are \(\mu_{1} = \mu_{2} = 25.698, \mu_{3} = 10.0907, \mu_{4} = 5.9214, \mu_{5} = 3.4321\).}
	\label{tab 3}
	\begin{tabular}{crrrrrc}
		\hline
		\multicolumn{7}{c}{\qquad\qquad Confidence levels}\\
		\cline{2-7}
		\( n \) & \multicolumn{1}{c}{0.99} & \multicolumn{1}{c}{0.95} & \multicolumn{1}{c}{0.90} & \multicolumn{1}{c}{0.85} &
		\multicolumn{1}{c}{0.80} &
		\multicolumn{1}{c}{0.75}\\
		\hline
		100	&0.999	&0.991	&0.972	&0.939	&0.906	&0.858\\			
		&0.003	&0.015	&0.035	&0.059	&0.089	&0.114\\
		300	&0.999	&0.981	&0.950	&0.919	&0.873	&0.816\\            
		&0.003	&0.023	&0.053	&0.075	&0.114	&0.144\\
		500	&0.998	&0.977	&0.947	&0.914	&0.867	&0.820\\            
		&0.005	&0.020	&0.041	&0.057	&0.087	&0.106\\
		1000	&0.992	&0.971	&0.937	&0.895	&0.855	&0.796\\           
		&0.010	&0.031	&0.061	&0.073	&0.105	&0.129\\
		2000	&0.990	&0.958	&0.911	&0.866	&0.824	&0.774\\           
		&0.006	&0.016	&0.024	&0.034	&0.052	&0.055\\
		3000	&0.989	&0.950	&0.897	&0.852	&0.795	&0.749\\           
		&0.004	&0.022	&0.034	&0.049	&0.061	&0.064
		\\ \hline
	\end{tabular}
\end{table}

In all three examples we observe the same paterns.
The bootstrap procedure mimics well the most of the underlying distribution of 
\( n\|\PS_{1} - \PP_{1}\|_{2}^{2} \).
For a really small sample size \( n=100 \), there is a problem of approximating the high quantiles,
while for \( n \) of order 300 or larger, it works surprisingly well in different setups and for different
dimensions \( p \) including the case with \( p > n \).

\clearpage
\section{Proofs}\label{SproofsPP}
This section presents the proof of the main theorem as well as some further statements. Before going  to the proof we outline its  main steps. In Section~\ref{gap real world} we show that
\begin{EQA}
	\text{\(\X\) - world: } 
	&\qquad &
	n\|\PS_{r} - \PP_{r}\|_{2}^{2}  \approx \|\xi\|_{2}^{2}, \quad \xi \sim \ND(0, \Gamma_{r}),
\end{EQA}
where \(\Gamma_{r}\) defined in (\refeq{eq:cov-matrix-gamma}). 
Further, in Section~\ref{gap bootstrap world} we demonstrate that the similar relation holds in the bootstrap world, namely
\begin{EQA}
	\text{Bootstrap world: } \quad 
	n\|\PB_{r} - \PS_{r}\|_{2}^{2}  
	&\approx &
	\|\xib\|_{2}^{2}, \quad \xib \sim \ND(0, \Gammab_{r}),
\end{EQA}
where  \(\Gammab_{r}\) is defined below in~(\refeq{eq: cov matrix gamma bootstrap}). 
To compare \(\xi\) and \(\xib\)  we apply Gaussian comparison inequality, Lemma~\ref{l: explicit gaussian comparison}. The details are in Section~\ref{gauss comparison}.  All necessary concentration inequalities for sample covariances in the \(\X\) - world and bootstrap world may be found in the Appendix~\ref{SappendPP} and Section~\ref{concentration bootstrap world} respectively. 

In all our results, we implicitly assume 
\begin{EQA} 
\label{eq: cond}
	\frac{\Tr \CM}{\gu_{r}} \left(\sqrt{\frac{\log n}{n}} + \sqrt{\frac{\log p}{n}} \right) 
	&\lesssim &
	1 .
\end{EQA}
Otherwise, the main result becomes trivial. 

\subsection{Concentration inequalities for covariance matrices and spectral projectors}\label{concentration bootstrap world}

\begin{theorem}\label{th: concentration for sample cov bootstrap world}
	Assume that the conditions of Theorem~\ref{th: main} hold. Then the following inequality holds with \(\Pb\)-probability at least \(1 - \frac{1}{n} \)
	\begin{EQA}
		\Pbb\left(
			\|\SCb - \SC\| \lesssim 
			\Tr \CM \left[ \sqrt{\frac{\log n}{n}} \bigvee \sqrt{\frac{\log p}{n}} \right]
		\right) 
		&\geq &
		1 - \frac{1}{n}.
	\end{EQA}
\end{theorem}
\begin{proof}
	We prove this theorem applying a combination of matrix concentration inequalities. For simplicity we denote \(\xi_{i} \eqdef w_{i} - 1\) and \(\A_{i} \eqdef X_{i} X_{i}^{\t}\) for all \(i = 1, \ldots, n\). It is easy to see that \(\SCb - \SC\) is a Gaussian matrix series. Indeed,
	\begin{EQA}
		\SCb - \SC 
		&=& 
		\frac{1}{n} \sum_{k=1}^{n} \xi_{k} \A_{k}.
	\end{EQA}
	Hence, to estimate \(\|\SCb - \SC \|\) we may directly apply Lemma~\ref{matrix series}, which gives us that 
	\begin{EQA}
		\label{eq: matrix concentration 1}
		\Pbb\left(
			\left\|\frac{1}{n}\sum_{k=1}^{n} \xi_{k} \A_{k} \right\| 
			\lesssim 
			\sigma \frac{\sqrt{\log n} + \sqrt{\log p}}{n}\right) 
		&\geq &
		1- \frac{1}{n},
	\end{EQA}
	where \(\sigma^{2} \eqdef \big\|\sum_{k=1}^{n} \A_{k}^{2} \big\|\). To finish the proof it remains to estimate with high \(\Pb\)-probability the variance parameter \(\sigma\). This may be done by using the Bernstein matrix concentration inequality, Lemma~\ref{matrix bernstein}. To proceed we need to check all  assumptions of Lemma~\ref{matrix bernstein}. Applying Lemma~\ref{th: Rosenthal for sub-exp from 1 to 2} with \(p = 2\) we may show that
	\begin{EQA}
		\label{eq: expectation of Ak2}
		\E \|\A_{k}\|^{2} \leq \E\|X_{k}\|^{4} 
		&\lesssim &
		\Tr^{2} \CM.
	\end{EQA}
	Moreover, application of the same lemma with \(p \asymp \log^{2} n\) gives us that
	\begin{EQA}
		\label{eq:tail-of-Ak2}
		\Pb(\|\A_{k}\|^{2} \lesssim  \Tr^{2}(\CM) \log^{2} n ) 
		&\geq & 
		1 - \frac{1}{n}.
	\end{EQA}
	Introduce the following event:
	\begin{EQA}
		\Eset_{1} 
		&\eqdef &		
		\bigl\{ 
			\max_{1 \leq k \leq n} \|\A_{k}^{2} - \E \A_{k}^{2}\| \lesssim \Tr^{2}(\CM) \log^{2} n
		\bigr\}.
	\end{EQA}
	It follows from~(\refeq{eq: expectation of Ak2})--(\refeq{eq:tail-of-Ak2}) and the union bound that \(\Pb(\Eset_{1}^{c} ) \leq \frac{1}{n}\). Introduce the following variance parameter
	\begin{EQA}
		\sigmat^{2} 
		&\eqdef &
		\big\|\sum_{k=1}^{n} \E (\A_{k}^{2} - \E \A_{k}^{2})^{2} \big\|.
	\end{EQA}
	Analogously to~(\refeq{eq: expectation of Ak2}) one may show that \(\sigmat^{2} \leq n \E\|\A_{1}\|^{4} \leq n \E\|X\|^{8} \lesssim n  \Tr^4 \CM \). Applying Lemma~\ref{matrix bernstein} we get
	\begin{EQA}
		\label{eq: concentration of A_{k}^{2}}
		&&
		\Pb \left( \big\|\sum_{k=1}^{n} (\A_{k}^{2} - \E \A_{k}^{2})\big\| \gtrsim \sqrt{n} \Tr^{2}(\CM) (\sqrt{\log n}+ \sqrt{\log p}) \right)  \\
		& \leq &
		\Pb \left( \big\|\sum_{k=1}^{n} (\A_{k}^{2} - \E \A_{k}^{2})\big\|  \gtrsim \sqrt{n} \Tr^{2}(\CM) (\sqrt{\log n}+ \sqrt{\log p}), \Eset_{1} \right) + \frac{1}{n}  \\
		&\leq& \frac{2}{n}.  
	\end{EQA}
	Combining~(\refeq{eq: expectation of Ak2}) and~(\refeq{eq: concentration of A_{k}^{2}}) we may write that with \(\Pb\)-probability at least \(1 - \frac{1}{n}\)
	\begin{EQA}
		\sigma^{2} 
		&\lesssim &
		n \Tr^{2}\CM + \sqrt{n}(\sqrt{\log n}
		+ \sqrt{\log p}) \Tr^{2}\CM  \lesssim n \Tr^{2} \CM.
	\end{EQA}
	Substituting the last inequality to~(\refeq{eq: matrix concentration 1}) we finish the proof of this theorem.
\end{proof}

Let us introduce the following notations
\begin{EQA}
	\EEb 
	&\eqdef &
	\SCb - \CM, \quad \EEhb \eqdef \SCb - \SC, \quad \EEh = \SC - \CM.
\end{EQA}
Denote
\begin{EQA}
	L_{r}(\EEhb) 
	&\eqdef &
	\PP_{r} (\SCb - \SC) \CC_{r} + \CC_{r} (\SCb - \SC) \PP_{r},
\end{EQA}
where 
\begin{EQA}
	\CC_{r} 
	&\eqdef &
	\sum_{s  \neq r} \frac{1}{\mu_{r} - \mu_{s}} \PP_{s}.
\end{EQA}

\begin{theorem}[Concentration results in the bootstrap world]
	\label{th: concentration bootstrap world}
	Assume that the conditions of Theorem~\ref{th: main} hold. Then the following bound holds with \(\Pb\)-probability at least \(1 - \frac{1}{n}\)
	\begin{EQA}
		\label{eq: bound 3b}
		\Pbb\left(
			\bigl|\|\PB_{r} - \PS_{r}\|_{2}^{2} - \|L_{r}(\EEhb)\|_{2}^{2} \bigr| \lesssim  \Delta 
		\right) 	
		&\geq &
		1 - \frac{1}{n},
	\end{EQA}
	where
	\begin{EQA}
		\label{eq: delta def}
		\Delta 
		&\eqdef &		
		m_{r} \frac{\Tr^3\CM)}{\gu_{r}^{3}} 
		\left[ \frac{\log n}{n} \bigvee \frac{\log p}{n} \right]^{3/2}. 
	\end{EQA}
\end{theorem}

\begin{proof}
	Applying Lemma~\ref{decomposition real world} we may write 
	\begin{EQA}
		\PB_{r} - \PS_{r} 
		&=& 
		\PB_{r} - \PP_{r} - (\PS_{r} - \PP_{r}) = L_{r}(\EEb) - L_{r}(\EEh) + S_{r}(\EEb) + S_{r}(\EEh),
	\end{EQA}
	where
	\begin{EQA}
		\label{eq: Sr bound}
		\|S_{r}(\EEh)\| &\leq &
		14 \left(\frac{\|\EEh\|}{\gu_{r}} \right)^{2}, \quad	
		\|S_{r}(\EEb)\| \leq 14 \left(\frac{\|\EEb\|}{\gu_{r}} \right)^{2}.
	\end{EQA}
	It is easy to see that
	\begin{EQA}
		L_{r}(\EEb) - L_{r}(\EE) &=& L_{r}(\EEhb).
	\end{EQA}
	Let us also denote \(S_{r}(\EEhb) \eqdef S_{r}(\EEb) + S_{r}(\EEh)\). 
	We may rewrite the difference \(\|\PB_{r} - \PS_{r}\|_{2}^{2} - \|L_{r}(\EEhb)\|_{2}^{2}\) in the following way:
	\begin{EQA}
		\|\PB_{r} - \PS_{r}\|_{2}^{2} - \|L_{r}(\EEhb)\|_{2}^{2} 
		&=& 
		2\langle L_{r}(\EEhb), S_{r}(\EEhb )\rangle +\|S_{r}(\EEhb ) \|_{2}^{2}.
	\end{EQA}
	Applying the Cauchy--Schwarz inequality we get
	\begin{EQA}
		\label{eq: Cauchy--Schwarz inequality for difference}
		\left| \|\PB_{r} - \PS_{r}\|_{2}^{2} - \|L_{r}(\EEhb)\|_{2}^{2} \right| 
		&\leq &
		2\|L_{r}(\EEhb)\|_{2} \| S_{r}(\EEhb )\|_{2} + \|S_{r}(\EEhb ) \|_{2}^{2}.
	\end{EQA}
	It follows from~(\refeq{eq: Sr bound})
	\begin{EQA}
		\|S_{r}(\EEb)\| 
		&\lesssim & 
		\left( \frac{\|\EEh\|}{\gu_{r}} \right)^{2} +   \left(\frac{\|\EEb\|}{\gu_{r}} \right)^{2}
		\lesssim 
		\left(\frac{\|\EEh\|}{\gu_{r}} \right)^{2} +   \left(\frac{\|\EEhb\|}{\gu_{r}} \right)^{2}.
	\end{EQA}
	From Theorems~\ref{th: concentration for sample cov bootstrap world}, 
	\ref{th: concentration for sample cov real world}, and condition~(\refeq{eq: cond}) we may assume that without loss of generality that the following inequality holds
	\begin{EQA}
		\max\{\|\EEb\|, \|\EEh\|\} 
		&\leq &
		\frac{\gu_{r}}{2}.
	\end{EQA}
	This fact guarantees that \(\Rank \PB_{r} = \Rank \PS_{r} = \Rank \PP_{r} = m_{r}\).  
	Applying~(\refeq{eq: Sr bound}) and Theorems~\ref{th: concentration for sample cov bootstrap world},~\ref{th: concentration for sample cov real world} we get that with \(\Pb\)-probability at least \(1 - \frac{1}{n}\):
	\begin{EQA}
		\label{eq: S_{r} concentration}
		\Pbb \left(
			\|S_{r}(\EEhb ) \|_{2} \lesssim \sqrt{m_{r}} \frac{\Tr^{2} \CM }{\gu_{r}^{2}} 
			\left[ \frac{\log n}{n} \bigvee \frac{\log p}{n} \right] 
		\right) 
		&\geq &
		1 - \frac{1}{n}.
	\end{EQA}
	It remains to estimate \(\|L_{r}(\EEhb)\|_{2}\). We proceed similarly to the proof of Theorem~\ref{th: concentration for sample cov bootstrap world}. We get that with \(\Pb\)-probability at least \(1 - \frac{1}{n}\):
	\begin{EQA}
		\Pbb \left(
			\| L_{r}(\EEhb)\|_{2} \lesssim \sqrt{m_{r}} \frac{\Tr \CM}{\gu_{r}} 
			\left[\sqrt{\frac{\log n}{n}} \bigvee \sqrt{\frac{\log p}{n}} \right]
		\right) 
		&\geq &
		1 - \frac{1}{n}.
	\end{EQA}
	From the last bound and inequalities~(\refeq{eq: Cauchy--Schwarz inequality for difference})--(\refeq{eq: S_{r} concentration}) we conclude that  with \(\Pb\)-probability at least \(1 - \frac{1}{n}\):
	\begin{EQA}
		\Pbb \left( 
			\|\PB_{r} - \PS_{r}\|_{2}^{2} - \|L_{r}(\EEhb)\|_{2}^{2} \lesssim \Delta_{1} \right) 
		&\geq &
		1 - \frac{1}{n}, 
	\end{EQA}
	where
	\begin{EQA}
		\Delta_{1}^{*} 
		& \eqdef &
		m_{r} \frac{\Tr^3 \CM}{\gu_{r}^{3}} 
		\left[ \frac{\log n}{n} \bigvee \frac{\log p}{n}  \right] \left[\sqrt{\frac{\log n}{n}} \bigvee \sqrt{\frac{\log p}{n}}\right] + m_{r} \frac{\Tr^4 \CM }{\gu_{r}^{4}} \left[ \frac{\log n}{n} \bigvee \frac{\log p}{n} \right]^{2}.
	\end{EQA}
	Applying condition~(\refeq{eq: cond}) we get that
	\begin{EQA}
		\Delta_{1}^{*} 
		&\leq &
		m_{r} \frac{\Tr ^3\CM)}{\gu_{r}^{3}} 
		\left[ \frac{\log n}{n} \bigvee \frac{\log p}{n} \right]^{3/2}.
	\end{EQA}
\end{proof}

\subsection{Approximation in the \(\X\) - world}\label{gap real world}

The main result of this section is the following theorem.

\begin{theorem}\label{th: GAP real world}
	Assume that the conditions of Theorem~\ref{th: main} hold. Let   \(\xi \sim \ND(0, \Gamma_{r})\), where \(\Gamma_{r}\) is defined in~(\refeq{eq:cov-matrix-gamma}). Then for all \(x: x > 0\)  the following bounds hold
	\begin{EQA}
		\Pb(n\|\PS_{r} - \PP_{r}\|_{2}^{2} > x) 
		& \leq &
		\Pb(\|\xi\|_{2}^{2} \geq x_{-}) + \err_{1}, \\
		\Pb(n\|\PS_{r} - \PP_{r}\|_{2}^{2} > x) 
		& \geq &
		\Pb(\|\xi\|_{2}^{2} \geq x_{+}) - \err_{1},
	\end{EQA}
	where \(x_\pm \eqdef x \pm \err_{2}\) and
	\begin{EQA}
		\err_{1} 
		& \eqdefapp &  
		m_{r}^{1/2} \frac{ \Tr \Gamma_{r}}{\sqrt{\lambda_{1}(\Gamma_{r}) \lambda_{2}(\Gamma_{r})}} \left( \sqrt{\frac{\log n}{n}} + \sqrt{\frac{\log p}{n}} \right), \\
		\err_{2} 
		& \eqdefapp &
		m_{r} \frac{\Tr^3\CM}{\gu_{r}^{3}} \sqrt{\frac{\log^3 n}{ n}}.
	\end{EQA}
\end{theorem}

\begin{proof}[Proof of Theorem~\ref{th: GAP real world}]
	Let us fix an arbitrary \(x \geq 0\). Without loss of generality we may assume that \(\err_{1} \lesssim 1\). Otherwise the claim is trivial. This fact implies that the condition~(\refeq{eq: real world cond}) holds.  
	
	Let us rewrite \(\PS_{r} - \PP_{r}\) as follows
	\begin{EQA}
		n\|\PS_{r} - \PP_{r}\|_{2}^{2} 
		&=& 
		2n \| \PP_{r} \EEh \CC_{r}\|_{2}^{2} 
		+ n\|\PS_{r} - \PP_{r}\|_{2}^{2} - 2n\| \PP_{r} \EEh \CC_{r}\|_{2}^{2}.
	\end{EQA}
	Theorem~\ref{th: concentration real world} implies that with probability at least \(1  - \frac{1}{n}\) 
	\begin{EQA}
		\bigl|n\|\PS_{r} - \PP_{r}\|_{2}^{2} - 2n\| \PP_{r} \EEh \CC_{r}\|_{2}^{2}\bigr| 
		&\leq &
		\Delta_{1}^* \eqdefapp  m_{r} \frac{\Tr^3 \CM }{\gu_{r}^{3}} \sqrt{\frac{\log^3 n}{ n}}.
	\end{EQA}
	Hence, we may write down the following two-sided inequalities
	\begin{EQA}
		\Pb(2n \| \PP_{r} \EEh \CC_{r}\|_{2}^{2} \geq x + \Delta_{1}^{*}) - \frac{1}{n}
		& \leq & 
		\Pb(n\|\PS_{r} - \PP_{r}\|_{2}^{2} > x) \\
		& \leq & 
		\Pb(2n \| \PP_{r} \EEh \CC_{r}\|_{2}^{2} \geq x - \Delta_{1}^{*})  + \frac{1}{n}.
	\end{EQA}
	For simplicity we denote \(x_{\pm} \eqdef x_{\pm} \eqdef x \pm \Delta_{1}^*\). Without loss of generality, we consider the case of the upper bound only, i.e. we set \(z \eqdef x_{-}\). Similar calculations are valid for \(x_{+}\).
	
	Let \(\{\ee_{j}\}_{j = 1}^{p}\) be an arbitrary orthonormal basis in \(\R^{p}\). Denote by \(\Psi_{kl} \eqdef \ee_{k} \ee_l^{\T}, l, k = 1, \ldots , p\). 
	Then \(\{\Psi_{kl}\}_{k,l=1}^{p}\) is the orthonormal basis in \(\R^{p\times p}\) with respect to the scalar product given by \(\langle\A, \B\rangle \eqdef \Tr \A \B^{\T}, \A, \B \in \R^{p \times p}\).  
	By Parseval's identity
	\begin{EQA}
		2n \| \PP_{r} \EEh \CC_{r}\|_{2}^{2} 
		&=& 
		2n \sum_{l,k = 1}^{p}  \langle\PP_{r} \EE \CC_{r}, \Psi_{kl}\rangle^{2} 
		= 
		2 n \sum_{l,k = 1}^{p} \langle\PP_{r} \EE \CC_{r} \ee_l, \ee_{k}\rangle^{2}.
	\end{EQA}
	We may set \(\ee_{j} \eqdef \uu_{j}\). Taking into account definition of \(\PP_{r}\) and \(\CC_{r}\) the last equation may be rewritten as follows 
	\begin{EQA}
		2n \| \PP_{r} \EEh \CC_{r}\|_{2}^{2} 
		&=& 
		2 n \sum_{k \in \Delta_{r}} \sum_{s \neq r} \sum_{l \in \Delta_{s}} 
		\langle\PP_{r} \EEh \CC_{r} \uu_l, \uu_{k}\rangle^{2}.
	\end{EQA}
	Let us fix arbitrary \(\uu_{k}, k \in \Delta_{r}\) and \(\uu_l, l \in \Delta_{s}, s \neq r\). For simplicity we denote them by \(\uu\) and \(\vv\) respectively. Then
	\begin{EQA}
		S(\uu, \vv) 
		&\eqdef &
		\sqrt{2n}\langle \PP_{r} \EEh \CC_{r} \vv, \uu\rangle 
		= 
		\sqrt{\frac{2}{n}} \sum_{i=1}^{n} \langle\uu, \PP_{r} X_{i}\rangle \langle\CC_{r} X_{i}, \vv\rangle.
	\end{EQA}
	It is easy to see that \( \langle  \uu, \PP_{r} X_{i} \rangle\) is a Gaussian r.v. with zero mean and variance 
	\( \E \langle  \uu, \PP_{r} X_{i} \rangle^{2} = \langle \uu, \PP_{r} \CM \PP_{r} \uu \rangle  = \mu_{r} \). 
	Then \( \langle  \uu, \PP_{r} X_{i} \rangle 	\myeq  \sqrt{ \mu_{r}} \etau_{\uu, i}\), where 
	\( \etau_{\uu, i}, i = 1, \ldots, n\) are i.i.d. \( \ND(0,1) \). 
	Similarly we may write that 
	\( \langle\CC_{r} X_{i}, \vv\rangle \myeq \sqrt{\frac{\mu_{s}}{(\mu_{r} - \mu_{s})^{2}}} \eta_{\vv, i}\), where \( \eta_{\uu, i}, i = 1, 
	\ldots, n\) are i.i.d. \( \ND(0,1) \). Hence, we obtain
	\begin{EQA}
		S(\uu, \vv) 
		&\myeq &
		\frac{1}{\sqrt{n}} \sum_{i=1}^{n} \sqrt{\frac{2\mu_{s} \mu_{r}}{(\mu_{r} - \mu_{s})^{2}}} \etau_{\uu, i} \eta_{\vv, i}.
	\end{EQA}
	Let us fix another pair \(\tilde{\uu}, \tilde{\vv}\) and investigate the covariance
	\begin{EQA}
		\Gamma((\uu, \vv),  (\tilde{\uu}, \tilde{\vv})) 
		&\eqdef &
		\Cov(S(\uu, \vv), S( \tilde{\uu},  \tilde{\vv})).
	\end{EQA}
	It is straightforward to check that
	\begin{EQA}
		\Gamma((\uu, \vv),  (\tilde{\uu}, \tilde{\vv})) 
		&=& 
		2\langle\CM \PP_{r} \uu, \PP_{r} \tilde{\uu} \rangle \langle\CM \CC_{r} \vv, \CC_{r} \tilde{\vv} \rangle 
		= 
		2 \Gamma_{1}(\uu, \tilde{\uu}) \Gamma_{2}(\vv, \tilde{\vv}),
	\end{EQA}
	where for simplicity we denoted 
	\begin{EQA}
		\Gamma_{1}(\uu, \tilde{\uu}) 
		&\eqdef &
		\langle\CM \PP_{r} \uu, \PP_{r} \tilde{\uu} \rangle, \quad \Gamma_{2}(\vv, \tilde{\vv}) 
		\eqdef 
		\langle\CC_{r} \CM \CC_{r} \vv,  \tilde{\vv} \rangle.
	\end{EQA}  
	Moreover, direct calculations yield that
	\begin{EQA}
		\label{def: gamma1 and gamma2}
		\Gamma_{1}(\uu, \tilde{\uu}) 
		&=& 
		\begin{cases}
			0, & \text{ if } \uu \neq \tilde{\uu}, \\
			\mu_{r}, & \text{ if } \uu = \tilde{\uu},
		\end{cases}
		\quad
		\Gamma_{2}(\vv, \tilde{\vv}) = 
		\begin{cases}
			0, & \text{ if } \vv \neq \tilde{\vv},  \\
			\frac{\mu_{s}}{(\mu_{r} - \mu_{s})^{2}}, & \text{ if } \vv = \tilde{\vv}.
		\end{cases}
	\end{EQA}
	We may think of \(S_{r} \eqdef (S(\uu_{k}, \uu_l), k \in \Delta_{r}, s \neq r, l \in \Delta_{s})\) as a random vector in the dimension \(d \eqdef m_{r} \sum_{s \neq r} m_{s} \) (it is easy to see that \( d \asymp p\)) with the following covariance matrix \(\Gamma_{r}\) (compare with~(\refeq{eq:cov-matrix-gamma})):
	\begin{EQA}	
		\Gamma_{r} 
		&\eqdef &
		\begin{pmatrix}
			\Gamma_{r 1} & \OO & \ldots & \OO \\
			\OO & \Gamma_{r 2} & \OO \ldots & \OO\\
			\ldots \\
			\OO & \ldots & \OO & \Gamma_{r q}
		\end{pmatrix},
	\end{EQA}
	where \( \Gamma_{rs} = \frac{2\mu_{r} \mu_{s}}{(\mu_{r} - \mu_{s})^{2}} \II_{m_{r} m_{s}}, s \neq r \),
	are diagonal matrices of order \( m_{r} m_{s}\times m_{r} m_{s} \) with values \( \frac{2\mu_{r} \mu_{s}}{(\mu_{r} - \mu_{s})^{2}} \) on the main diagonal. In these notations we may write
	\begin{EQA}
		\Pb\left( 2n\|\PP_{r} \EEh \CC_{r}\|_{2}^{2} \geq z \right) 
		&=& 
		\Pb\left( \|S_{r}\|^{2} \geq z \right).
	\end{EQA}
	Since \(\PP_{r} X_{i}\) and \(\CC_{r} X_{i}\) are independent Gaussian vectors it is straightforward to check that 
	the conditional distribution of \(S_{r}\) with respect to \( \Yu = (\PP_{r} X_{1}, \ldots, \PP_{r} X_n)\) is Gaussian with zero mean and covariance matrix \(\Gamma_{r}^{\Yu} = \frac{1}{n} \sum_{i=1}^{n} \Gamma_{ri}^{\Yu}\), where 
	\(\Gamma_{ri}^{\Yu}  \eqdef [\Gamma_{ri}^{\Yu}((\uu_{k_{1}}, \uu_{l_{1}}), (\uu_{k_{2}}, \uu_{l_{2}})), k_{1}, k_{2} \in
	\Delta_{r}, l_{1} \in \Delta_{s_{1}}, l_{2} \in \Delta_{s_{2}}, s_{1}, s_{2} \neq r]\) and
	\begin{EQA}
		\Gamma_{ri}^{\Yu}((\uu_{k_{1}}, \uu_{l_{1}}), (\uu_{k_{2}}, \uu_{l_{2}}))
		&=&  
		2 \mu_{r} \etau_{\uu_{k_{1}},i} \etau_{\uu_{k_{2}},i} \Gamma_{2}(\uu_{l_{1}}, \uu_{l_{2}}). 
	\end{EQA}
	Due to~(\refeq{def: gamma1 and gamma2}) we conclude that \( \Gamma_{r}^{\Yu}((\uu_{k_{1}}, \uu_{l_{1}}), (\uu_{k_{2}}, \uu_{l_{2}})) = 0\) if \( l_{1} \neq l_{2}\). Let \(\Pb(\cdot\cond  \Yu)\) be the conditional probability w.r.t. \(\Yu\). 
	We show that \(\Pb(\|S_{r}\|^{2}\geq z \cond \Yu)\) may be approximated by \(\Pb(\|\xi\|^{2}  \geq z)\), where \(\xi \sim \ND(0, \Gamma_{r})\).  
	For this aim we may apply Corollarly~\ref{cor 1 gauss}. Hence, we need to check that 
	\(\| \Gamma_{r}^{-\frac{1}{2}} \Gamma_{r}^{\Yu} \Gamma_{r}^{-\frac{1}{2}} - \II \|\) is small.  
	Let us denote by \(\Pb_{\Yu}(\cdot)\) the distribution of \(\Yu\), i.e. 
	\(\Pb_{\Yu}(A) = \Pb(\Yu \in A), A \in \mathfrak B(\R^{p})\). 
	We also introduce the following event
	\begin{EQA}
		\Eset_{1}(\delta)  
		&\eqdef &
		\bigl\{ 
		\Yu \colon \| \Gamma_{r}^{-\frac{1}{2}} \Gamma_{r}^{\Yu} \Gamma_{r}^{-\frac{1}{2}} - \II \|
		\leq \delta
		\bigr\}, 
		\quad \delta > 0.
	\end{EQA}
	If \( \max_{1 \le k \le n} \| \Gamma_{r}^{-\frac{1}{2}} \Gamma_{rk}^{\Yu} \Gamma_{r}^{-\frac{1}{2}} - \II \| \le R \) for some \( R = R(n, \Gamma_{r})\), then it follows from Lemma~\ref{matrix bernstein} that
	\begin{EQA}
		\label{eq: berstein inequality 1}
		\Pb_{\Yu}\left(\left\| \frac{1}{n}\sum_{i=1}^{n} \bigl( \Gamma_{r}^{-\frac{1}{2}} \Gamma_{rk}^{\Yu} \Gamma_{r}^{-\frac{1}{2}} - \II\bigr) \right\| \lesssim \frac{s}{n} 
		\right) 
		&\geq &
		1 - d \cdot \exp\bigg(-\frac{s^{2}}{\sigma^{2}}\bigg),
	\end{EQA}
	provided that \( R s \lesssim \sigma^{2}\), where 
	\begin{EQA}
		\sigma^{2} 
		&\eqdef &
		\left\| 
		\sum_{i=1}^{n} \E_{\Yu} \bigl( \Gamma_{r}^{-\frac{1}{2}} \Gamma_{ri}^{\Yu} \Gamma_{r}^{-\frac{1}{2}} - \II \bigr)^{2} 
		\right\|.
	\end{EQA}
	It is straightforward to check that \(\Gamma_{r}^{-\frac{1}{2}} \Gamma_{ri}^{\Yu} \Gamma_{r}^{-\frac{1}{2}}\) is a block-diagonal matrix. The number of blocks equals \( \sum_{s \neq r} \Delta_{s} \), all of them are the same and have the following structure 
	\begin{EQA}	
		&&
		\begin{pmatrix}
			\etau_{\uu_{k_{1}},i}^{2} & \etau_{\uu_{k_{1}},i} \etau_{\uu_{k_{2}},i} & \ldots & \etau_{\uu_{k_{1}},i} \etau_{\uu_{k_{r}},i}  \\
			\etau_{\uu_{k_{1}},i} \etau_{\uu_{k_{2}},i} & \etau_{\uu_{k_{2}},i}^{2} & \ldots & \etau_{\uu_{k_{2}},i} \etau_{\uu_{k_{r}},i}\\
			&&\ldots &\\
			\etau_{\uu_{k_{1}},i} \etau_{\uu_{k_{r}},i} & \etau_{\uu_{k_{2}},i} \etau_{\uu_{k_{r}},i} & \ldots  &  \etau_{\uu_{k_{r}},i}^{2}
		\end{pmatrix},
	\end{EQA}
	where \( k_{j} \in \Delta_{r}, j = 1, \ldots, r \). Hence, 
	\begin{EQA}
		\| \Gamma_{r}^{-\frac{1}{2}} \Gamma_{ri}^{\Yu} \Gamma_{r}^{-\frac{1}{2}} - \II \|  
		&\le& 
		\| \Gamma_{r}^{-\frac{1}{2}} \Gamma_{ri}^{\Yu} \Gamma_{r}^{-\frac{1}{2}} \|_{2} + 1 
		= 
		\bigg(\sum_{k_{1}, k_{2} \in \Delta_{r} } \etau_{\uu_{k_{1}},i}^{2} \etau_{\uu_{k_{2}},i}^{2}\bigg)^{\frac{1}{2}} + 1 
		\\
		&=& 
		\sum_{k \in \Delta_{r} } \etau_{\uu_{k},i}^{2} + 1.
	\end{EQA}
	Applying Lemma~\ref{th: Rosenthal for sub-exp from 1 to 2} we obtain that 
	\begin{EQA}
		\Pb_{\Yu} \left(\| \Gamma_{r}^{-\frac{1}{2}} \Gamma_{ri}^{\Yu} \Gamma_{r}^{-\frac{1}{2}} -\II \| \lesssim m_{r} \log n \right) 
		&\geq &
		1 - \frac{1}{n}.
	\end{EQA}
	Moreover, let \( R \eqdefapp m_{r} \log n\). Denote \( \Eset_{2} \eqdef \{ \max_{1 \le i \le n}\| \Gamma_{r}^{-\frac{1}{2}} \Gamma_{ri}^{\Yu} \Gamma_{r}^{-\frac{1}{2}} -\II \| \le R  \}\). Then, \( \Pb (\Eset_{2}) \geq 1 - \frac{1}{n} \).
	
	Let us estimate \(\sigma^{2}\). We fix \(k_{1}, k_{2} \in \Delta_{r}, l_{1} \in \Delta_{s_{1}}, l_{2} \in \Delta_{s_{2}}, s_{1}, s_{2} \neq r \). Direct calculation gives us that
	\begin{EQA}
		&& \nquad
		\E_{\Yu} [\Gamma_{r}^{-1/2}\Gamma_{ri}^{\Yu}\Gamma_{r}^{-1/2}]^{2}((\uu_{k_{1}}, \uu_{l_{1}}), (\uu_{k_{2}}, \uu_{l_{2}})) \\
		&=& 
		\sum_{s \neq r} \sum_{k \in \Delta_{r}} \sum_{l \in \Delta_{s}} \E_{\Yu}  \etau_{\uu_{k_{1}}, i} \etau_{\uu_{k,i}}^{2} \etau_{\uu_{k_{2}}, i} \Gamma_{2}(\uu_{l_{1}}, \uu_{l}) \Gamma_{2}(\uu_{l}, \uu_{l_{2}}) 
		\\
		&& \qquad
		\times \sqrt{\frac{(\mu_{r} - \mu_{s_{1}})^{2} (\mu_{r} - \mu_{s_{2}})^{2}}{\mu_{s_{1}} \mu_{s_{2}}} } \frac{(\mu_{r} - \mu_{s})^{2}}{\mu_{s}}.
	\end{EQA}
	Let \((\uu_{k_{1}}, \uu_{l_{1}}) \neq (\uu_{k_{2}}, \uu_{l_{2}}) )\). 
	Then it is easy to check that 
	\begin{EQA}
		\E_{\Yu} [\Gamma_{r}^{-1/2}\Gamma_{ri}^{\Yu}\Gamma_{r}^{-1/2}]^{2}((\uu_{k_{1}}, \uu_{l_{1}}), (\uu_{k_{2}}, \uu_{l_{2}})) 
		&=& 0.
	\end{EQA} 
	This means that it is a diagonal matrix. 
	Assume now that \((\uu_{k_{1}}, \uu_{l_{1}}) = (\uu_{k_{2}}, \uu_{l_{2}}) )\). 
	Then
	\begin{EQA}
		\E_{\Yu} [\Gamma_{r}^{-1/2}\Gamma_{ri}^{\Yu}\Gamma_{r}^{-1/2}]^{2}((\uu_{k_{1}}, \uu_{l_{1}}), (\uu_{k_{1}}, \uu_{l_{1}})) 
		&=& 
		\sum_{k \in \Delta_{r}}  \E_{\Yu}  \etau_{\uu_{k_{1}},i} \etau_{\uu_{k,i}}^{2} \etau_{\uu_{k_{2}}, i} =
		m_{r} + 2. 
	\end{EQA}
	Hence, 
	\begin{EQA}
		\label{eq: var 1}
		\|\E_{\Yu} \bigl( \Gamma_{r}^{-\frac{1}{2}} \Gamma_{ri}^{\Yu} \Gamma_{r}^{-\frac{1}{2}} - \II \bigr)^{2}\| 
		&\asymp &
		m_{r} \text{ and } \sigma^{2} \asymp n m_{r}.  
	\end{EQA}
	Let us denote 
	\begin{EQA}
		\Delta_{2}^{*} 
		&\eqdefapp &
		\sqrt{m_{r}} \left(\sqrt{\frac{\log n}{n}}  + \sqrt{\frac{\log p}{n}} \right).
	\end{EQA}
	It follows from~(\refeq{eq: berstein inequality 1}) and~(\refeq{eq: var 1}) that
	\begin{EQA}
		\Pb_{\Yu}\Bigl(\Eset_{1}^c(\Delta_{2}^{*})) \le \Pb_{\Yu}(\Eset_{1}^c(\Delta_{2}^{*}) \cap \Eset_{2}
		\Bigr) 
		+ \frac{1}{n} 
		&\leq &
		\frac{2}{n}.
	\end{EQA}
	Similarly to the previous calculations we may also estimate the probability of the following event
	\begin{EQA}
		\Eset_{3}(\delta) \eqdef  \{\Yu: \| \Gamma_{r}^{\Yu}  - \Gamma_{r} \| \leq \delta\}, 
		&\quad &
		\delta > 0.
	\end{EQA}
	It follows from Lemma~\ref{matrix bernstein} that
	\begin{EQA}
		\label{eq: berstein inequality 2}
		\Pb_{\Yu}\left( 
		\left\|\frac{1}{n} \sum_{i=1}^{n} \bigl( \Gamma_{ri}^{\Yu} - \Gamma_{r} \bigr) \right\| \lesssim \frac{s}{n}
		\right) 
		&\geq &
		1 - p \cdot \exp\bigg(-\frac{s^{2}}{\tilde{\sigma}^{2}}\bigg),
	\end{EQA}
	where 
	\begin{EQA}
	\label{eq: var 2}
		\tilde{\sigma}^{2} 
		&\eqdef &
		\left\| \sum_{i=1}^{n} \E_{\Yu} \bigl( \Gamma_{ri}^{\Yu} - \Gamma_{r} \bigr)^{2} \right\| \asymp 
		n \max_{s \neq r} \frac{4\mu_{r}^{2} \mu_{s}^{2} (m_{r} + 2) }{(\mu_{r} - \mu_{s})^{4}} 
		\asymp
		n m_{r} \|\Gamma_{r}\|^{2}.
	\end{EQA}
	Here we applied the same arguments as above. Introduce  the following quantity
	\begin{EQA}
		\Delta_{3}^{*} 
		&\eqdefapp &
		\sqrt{m_{r}} \| \Gamma_{r}\| \left(\sqrt{\frac{\log n}{n}}  + \sqrt{\frac{\log p}{n}} \right).
	\end{EQA}
	It follows from~(\refeq{eq: berstein inequality 2}) and~(\refeq{eq: var 2}) that
	\begin{EQA}
		\Pb(\Eset_{3}(\Delta_{3}^*)) &\geq &
		1 - \frac{1}{n}.
	\end{EQA}
	Let us denote \(\Eset \eqdef \Eset_{1}(\Delta_{2}^*) \cap \Eset_{3}(\Delta_{3}^*)\). Without loss of generality we may assume that
	\begin{EQA}
		\Pb(\Eset) &\geq &
		1 - \frac{1}{n}.
	\end{EQA}
	To finish the proof we apply Corollary~\ref{cor 1 gauss} to obtain
	\begin{EQA}
		\Pb\left( \|S_{r}\|^{2} \geq z\right) 
		&=& 
		\int \Pb\left( \|S_{r}\|^{2} \geq z  \cond \Yu = y \right) \, d \Pb_{\Yu}(y) \\
		&=& 
		\int_{\Eset}  \Pb\left( \|S_{r}\|^{2} \geq z  \cond \Yu = y \right) \, d \Pb_{\Yu}(y) + \int_{\Eset^{c}}\Pb\left( \|S_{r}\|^{2} \geq z  \cond \Yu = y \right) \, d \Pb_{\Yu}(y) \\
		&=& 
		\Pb( \|\xi\|^{2} \geq z) + \Rerr_{n}, 
	\end{EQA}
	where  
	\begin{EQA}
		|\Rerr_{n}| 
		&\leq &
		\Delta_{4}^* \eqdefapp \frac{\sqrt{m_{r}} \Tr \Gamma_{r}}{\sqrt{\lambda_{1}(\Gamma_{r}) \lambda_{2}(\Gamma_{r})}}
		\left( \sqrt{\frac{\log n}{n}} + \sqrt{\frac{\log p}{n}} \right).
	\end{EQA}
	Hence, we proved the following bound
	\begin{EQA}
		\Pb\bigl(n\|\PS_{r} - \PP_{r}\|_{2}^{2} > x \bigr) 
		&\leq &
		\Pb\bigl( \|\xi\|_{2}^{2} \geq x_{-} \bigr) +  \Delta_{4}^{*}.
	\end{EQA}
	Comparing definition of \(\Delta_{4}\) and \(\Delta_{1}\) with \(\err_{1}\)  and \(\err_{2}\) resp. we get the claim of the theorem.
\end{proof}

\subsection{Approximation in the bootstrap world}\label{gap bootstrap world}

The main result of this section is the following theorem.
\begin{theorem}\label{th: GAP bootstrap world}
	Assume that the conditions of Theorem~\ref{th: main} hold.
	Let \(\xib \sim \ND(0, \Gammab_{r})\), where  \(\Gammab_{r}\) is defined below in~(\refeq{eq: cov matrix gamma bootstrap}). For all \(x: x>0\) the following bounds hold with \(\Pb\)-probability at least \(1 - \frac{1}{n}\): 
	\begin{EQA}
		&& 
		\Pbb(n\|\PB_{r} - \PS_{r}\|_{2}^{2} > x) \leq \Pbb(\|\xib\|_{2}^{2} \geq x_{-}) +  \frac{1}{n}, \\ 
		&&
		\Pbb(n\|\PB_{r} - \PS_{r}\|_{2}^{2} > x) \geq \Pbb(\|\xib\|_{2}^{2} \geq x_{+}) -  \frac{1}{n}.
	\end{EQA}
	Here, \(x_{\pm} \eqdef x \pm \err_{3}\) and  
	\begin{EQA}
		\err_{3}
		&\eqdefapp &
		m_{r} \frac{\Tr^3 \CM}{\gu_{r}^{3}} \sqrt{\frac{\log^{3} n}{n} + \frac{\log^{3} p}{n}}.
	\end{EQA}
\end{theorem}
\begin{proof}
	Let us fix an arbitrary \(x \geq 0\).  We introduce the following notations
	\begin{EQA}
		\EEb &\eqdef &
		\SCb - \CM, \quad \EEhb \eqdef \SCb - \SC.
	\end{EQA}
	and remind  \(\EEh = \SC - \CM\). Applying Lemma~\ref{decomposition real world} we may write 
	\begin{EQA}
		\PB_{r} - \PS_{r} 
		&=& 
		\PB_{r} - \PP_{r} - (\PS_{r} - \PP_{r}) = L_{r}(\EEb) - L_{r}(\EEh) + S_{r}(\EEb) + S_{r}(\EEh).
	\end{EQA}
	It is easy to see that
	\begin{EQA}
		L_{r}(\EEb) - L_{r}(\EEh) 
		&=& 
		\PP_{r} (\SCb - \SC) \CC_{r} + \CC_{r} (\SCb - \SC) \PP_{r} \eqdef L_{r}(\hat \EEb).
	\end{EQA}
	Then
	\begin{EQA}
		n\|\PB_{r} - \PS_{r}\|_{2}^{2} 
		&=& 
		n\| L_{r}(\EEhb)\|_{2}^{2} + n\|\PB_{r} - \PS_{r}\|_{2}^{2} - n\| L_{r}(\EEhb)\|_{2}^{2}.
	\end{EQA}
	It follows from Theorem~\ref{th: concentration bootstrap world} that with \(\Pb\) - probability at least \(1 - \frac{1}{n}\)
	\begin{EQA}
		\Pbb\left(
		\left| n\|\PB_{r} - \PS_{r}\|_{2}^{2} - n\| L_{r}(\EEhb)\|_{2}^{2} \right| \leq \Delta_{1}^* 	
		\right) 
		&\geq &
		1 - \frac{1}{n},
	\end{EQA}
	where
	\begin{EQA}
		\Delta_{1}^* 
		&\eqdefapp & 
		m_{r} \frac{\Tr^3 \CM}{\gu_{r}^{3}} \sqrt{\frac{\log^{3} n}{n} + \frac{\log^{3} p}{n}}.
	\end{EQA}
	Introduce the notation \( x_{\pm} \eqdef x_{\pm} \eqdef x \pm \Delta_{1}^*\). From the previous inequality we may conclude the following two-sided inequalities
	\begin{EQA}
		\Pbb(2n \| \PP_{r} \EEhb \CC_{r}\|_{2}^{2} \geq x_{+}) - \frac{1}{n} 
		&\leq &
		\Pbb(n\|\PB_{r} - \PS_{r}\|_{2}^{2} > x) 
		\leq 
		\Pbb(2n \| \PP_{r} \EEhb \CC_{r}\|_{2}^{2} \geq x_{-})  + \frac{1}{n}.
	\end{EQA}
	It follows that we need to estimate the term \( 2n\| \PP_{r} \EEhb \CC_{r}\|_{2}^{2} \). Without loss of generality, we consider the case of the upper bound only, i.e. we set \(z \eqdef x_{+}\). Similar calculations are valid for \(x_{-}\). 
	Analogously to the approximation in the \(\X \)-world we choose \(\{\uu_{j}\}_{j=1}^{p} \) as an orthonormal basis in \(\R^{p}\). 
	By Parseval's identity,
	\begin{EQA}
		2n \| \PP_{r} \EEhb \CC_{r}\|_{2}^{2} 
		&=& 
		2 n \sum_{l,k = 1}^{p} \langle\PP_{r} \EEhb \CC_{r} \uu_l, \uu_{k}\rangle^{2}.
	\end{EQA}
	Applying the orthogonality of \(\PP_{r}\) and \(\CC_{r}\) we obtain 
	\begin{EQA}
		2n \| \PP_{r} \EEhb \CC_{r}\|_{2}^{2} 
		&=& 
		2 n \sum_{k \in \Delta_{r}} \sum_{s \neq r} \sum_{l \in \Delta_{s}} 
		\langle\PP_{r} \EEhb \CC_{r} \uu_l, \uu_{k}\rangle^{2}.
	\end{EQA}
	Let us fix arbitrary \(\uu_{k}, k \in \Delta_{r}\) and \(\uu_l, l \in \Delta_{s}, s \neq r\). For simplicity we denote them by \(\uu\) and \(\vv\) respectively. We may write
	\begin{EQA}
		S^{\sbt(\uu, \vv)} 
		&\eqdef &
		\sqrt{2n}\langle \PP_{r} \EEhb \CC_{r} \vv, \uu\rangle 
		= 
		\sqrt{\frac{2}{ n}} \sum_{i=1}^{n} \eta_{i} \langle\uu, \Yu_{i}\rangle \langle\vv, Y_{i}\rangle,
	\end{EQA}
	where we denoted \(\eta_{i} \eqdef w_{i} - 1, \Yu_{i} \eqdef \PP_{r} X_{i}\) and 
	\(Y_{i} \eqdef \CC_{r} X_{i}\). 
	Since \(\eta_{i} \sim \ND(0, 1)\), then
	\begin{EQA}
		S^{\sbt(\uu, \vv)} 
		&\myeq &
		\xib(\uu, \vv) \sim \ND(0, \Varb(\xib(\uu,\vv))), \, \Varb(\xib(\uu,\vv)) 
		= 
		\frac{2}{n} \sum_{i=1}^{n} \langle\uu, \Yu_{i})^{2} (\vv, Y_{i}\rangle^{2}.
	\end{EQA}
	Let us fix another pair \(\tilde{\uu}, \tilde{\vv}\) and investigate the covariance
	\begin{EQA}
		\Gammab_{r}((\uu, \vv),(\tilde{\uu}, \tilde{\vv})) 
		&\eqdef &
		\Covb(\xib(\uu, \vv), \xib( \tilde{\uu},  \tilde{\vv})).
	\end{EQA}
	Direct calculations show that
	\begin{EQA}
		\Gammab_{r}((\uu, \vv),(\tilde{\uu}, \tilde{\vv})) 
		&=& 
		\frac{2}{n} \sum_{i=1}^{n} \langle\uu, \Yu_{i}\rangle\langle\tilde{\uu}, \Yu_{i}\rangle \langle\vv, Y_{i}\rangle\langle\tilde{\vv}, Y_{i}\rangle. 
	\end{EQA}
	We form the following covariance matrix
	\begin{EQA}
		\label{eq: cov matrix gamma bootstrap}
		\Gammab_{r} 
		&\eqdef &
		\left[\Gammab((\uu, \vv),  (\tilde{\uu}, \tilde{\vv})) \right]_{((\uu, \vv),(\tilde{\uu}, \tilde{\vv}))}.
	\end{EQA}
	Denote \(\xib \eqdef (\xib(\uu_{k}, \uu_l),  k \in \Delta_{r}, s \neq r, l \in \Delta_{s})\).  
	Then
	\begin{EQA}
		\Pbb\left( 2n\|\PP_{r} \EEhb \CC_{r}\|_{2}^{2} \geq z \right) 
		&=& 
		\Pbb\left( \|\xib\|^{2} \geq z \right).
	\end{EQA}
	Comparing definition of \(\Delta_{1}\) and \(\err_{3}\) we conclude the claim of the theorem.
\end{proof}

\subsection{Gaussian comparison}\label{gauss comparison}
In this section we prove the following Lemma.
\begin{lemma}\label{l: gaussian comparison, our  covariance matrices}
	Let \(\xi \sim \ND(0, \Gamma_{r})\) and \(\xib \sim \ND(0, \Gammab_{r})\), where \(\Gamma_{r}\) and \(\Gammab_{r}\) are defined in~(\refeq{eq:cov-matrix-gamma}),~(\refeq{eq: cov matrix gamma bootstrap}) respectively. Let \( \mm \) be defined by the relations 
	\begin{EQA}
		\lambda_{\mm}(\Gamma_{r}) 
		&\geq &
		\Tr \Gamma_{r} \left(\sqrt{\frac{\log n}{n}} + \sqrt{\frac{\log p}{n}} \right) 
		\geq 
		\lambda_{\mm+1}(\Gamma_{r}).
	\end{EQA}
	Then the following holds with \(\Pb\)-probability al least \(1 - \frac{1}{n}\):
	\begin{EQA}
		\sup_{x \geq 0} |\Pb(\|\xi\|_{2}^{2} \geq x) - \Pbb(\|\xib\|_{2}^{2} \geq x)| 
		&\leq &
		\err_{4},
	\end{EQA}
	where 
	\begin{EQA}
		\err_{4} 
		&\eqdefapp &
		\frac{ \mm\, \Tr \Gamma_{r}}{\sqrt{\lambda_{1}(\Gamma_{r}) \lambda_{2}(\Gamma_{r})}} \left(\sqrt{\frac{\log n}{n}} + \sqrt{\frac{\log p}{n}} \right) +  \frac{\Tr (\II - \Pi_\mm)\Gamma_{r}}{\sqrt{\lambda_{1}(\Gamma_{r}) \lambda_{2}(\Gamma_{r})}}.
	\end{EQA}
	Here \( \Pi_{\mm} \) is a projector on the subspace spanned by the eigenvectors of \(\Gamma_{r}\) corresponding to its largest \( \mm \) eigenvalues.
\end{lemma}
\begin{proof}
	Without loss of generality we may assume that \(\err_{4} \lesssim 1\). The proof is based on the application of Corollary~\ref{cor 2 gauss}. First we estimate \( \|\Gammab_{r} - \Gamma_{r}\| \). Introduce the following notations
	\begin{EQA}
		\Gammab_{ri} 
		&\eqdef &
		[\Gammab_{ri}((\uu_{k_{1}}, \uu_{l_{1}}), (\uu_{k_{2}}, \uu_{l_{2}})), k_{1}, k_{2} \in
		\Delta_{r}, l_{1} \in \Delta_{s_{1}}, l_{2} \in \Delta_{s_{2}}, s_{1}, s_{2} \neq r],
	\end{EQA}
	where 
	\begin{EQA}
		\Gammab_{ri}((\uu_{k_{1}}, \uu_{l_{1}}), (\uu_{k_{2}}, \uu_{l_{2}})) 
		&\eqdef & 
		2\sqrt{\frac{\mu_{r}^{2}}{(\mu_{r} - \mu_{s_{1}})^{2}(\mu_{r} - \mu_{s_{2}})^{2}}} \etau_{\uu_{k_{1}},i} \etau_{\uu_{k_{2}},i} \eta_{\uu_{l_{1}}, i} \eta_{\uu_{l_{2}}, i}.
	\end{EQA}
	In these notations we may rewrite \(\Gammab_{r}\) as follows
	\begin{EQA}
		\Gammab_{r}((\uu_{k_{1}}, \uu_{l_{1}}), (\uu_{k_{2}}, \uu_{l_{2}})) 
		&=& 
		\frac{1}{n} \sum_{i=1}^{n} \Gammab_{ri}((\uu_{k_{1}}, \uu_{l_{1}}), (\uu_{k_{2}}, \uu_{l_{2}})).
	\end{EQA}
	Due to Lemma~\ref{matrix bernstein} we need to show that there exists \(R = R(n, \Gamma_{r})\) such that
	\begin{EQA}\label{eq: 11}
		\max_{1 \le k \le n} \|\Gammab_{rk} - \Gamma_{r}\| 
		&\lesssim &
		R,
	\end{EQA}
	and estimate
	\begin{EQA}\label{eq: 12}
		\tilde{\sigma}^{2} 
		&=& 
		\Big\|\sum_{k=1}^{n} \E ( \Gammab_{rk} - \Gamma_{r})^{2} \Big\| 
		= 
		n \Big\| \E (\Gammab_{r1} - \Gamma_{r})^{2} \Big\|.
	\end{EQA}
	It is obvious that \(\|\Gammab_{ri} - \Gamma_{r}\| \le \|\Gammab_{ri}\| + \| \Gamma_{r}\| \).  Let \( Z_{rj} \eqdef ( \etau_{\uu_{k}, j} \eta_{\uu_l, j}, s \neq r, l \in \Delta_{s}, k \in \Delta_{r} )^{\t},  j = 1, \ldots, n\). Since \( \Gammab_{r1} = Z_{r1} Z_{r1}^{\t} \) we obtain  
	\begin{EQA}
		\|\Gammab_{r1}\|  
		&=& 
		\| Z_{r1}\|^{2} = 2 \sum_{s \neq r} \sum_{k \in \Delta_{r}}   
		\sum_{l \in \Delta_{s_{1}}} \frac{\mu_{s} \mu_{r} }{(\mu_{s} - \mu_{r})^{2}} \etau_{\uu_{k}}^{2} \eta_{\uu_{l}}^{2}.
	\end{EQA}
	Applying Lemma~\ref{th: Rosenthal for sub-exp from 1 to 2} we get 
	\begin{EQA}
		\Pb \left(\|\Gammab_{r1}\|  \lesssim  \log^{2} n \, \Tr \Gamma_{r} \right) 
		&\geq &
		1 - \frac{1}{n}.
	\end{EQA}
	Moreover, it is obvious that \( \|\Gamma_{r}\| \le \Tr \Gamma_{r}  \). To bound \(\max_{1 \leq i \leq n} \|\Gammab_{ri} - \Gamma_{r}\|\) we introduce the following event
	\begin{EQA}
		\Eset_{1}  
		&\eqdef &
		\left\{ 
		\max_{1 \leq i \leq n} \|\Gammab_{ri} - \Gamma_{r}\|
		\lesssim \log^{2} n \, \Tr  \Gamma_{r}. 
		\right\}.
	\end{EQA}
	Using the union bound we may show that \(\Pb(\Eset_{1}^{c} ) \leq n^{-1} \). It remains to estimate \( \tilde{\sigma}^{2} \). Since
	\(\tilde{\sigma}^{2} =  n \| \E (\Gammab_{r1})^{2} - \Gamma_{r}^{2} \|\) we first calculate \( \E (\Gammab_{r1})^{2} \). It follows that
	\begin{EQA}
		\E (\Gammab_{r1})^{2} 
		&=& 
		\E Z_{r1} Z_{r1}^{\t} Z_{r1} Z_{r1}^{\t} = \E \|Z_{r1}\|^{2} Z_{r1} Z_{r1}^{\t}.
	\end{EQA}
	 Let us fix some \( s_{1}, s_{2} \neq r, k_{1}, k_{2} \in \Delta_{r}, l_{1} \in \Delta_{s_{1}}, l_{2} \in \Delta_{s_{2}} \). Then the entry of \(\E (\Gammab_{r1})^{2}\) in the position \( ((\uu_{k_{1}}, \uu_{l_{1}}), (\uu_{k_{2}}, \uu_{l_{2}})) \) has the following form
	\begin{EQA}
		&&
		\E \frac{4\mu_{r} \sqrt{\mu_{s_{1}} \mu_{s_{2}}}  }{|\mu_{s_{1}} - \mu_{r}| |\mu_{s_{2}} - \mu_{r}|}    \etau_{\uu_{k_{1}}} \etau_{\uu_{k_{2}}} \eta_{\uu_{l_{1}}} \eta_{\uu_{l_{2}}} \sum_{s \neq r} \sum_{k \in \Delta_{r}} \sum_{l \in \Delta_{s}} \frac{\mu_{s} \mu_{r} }{(\mu_{s} - \mu_{r})^{2}} \etau_{\uu_k}^{2} \eta_{\uu_l}^{2},
	\end{EQA}
	where \( \etau_{\uu_k}, \eta_{\uu_l}, k \in \Delta_{r}, l \in \Delta_{s}, s \neq r\), are i.i.d. \(\ND(0,1)\) r.v. It is easy to check that all off-diagonal entries are equal zero and it remains to estimate diagonal entries only. We obtain 
	\begin{EQA}
		\E \frac{4\mu_{r} \mu_{s_{1}}}{(\mu_{s_{1}} - \mu_{r})^{2}} \,
		\etau_{\uu_{k_{1}}}^{2} \eta_{\uu_{l_{1}}}^{2} 
		\sum_{s \neq r} \sum_{k \in \Delta_{r}} \sum_{l \in \Delta_{s}} 
			\frac{\mu_{s} \mu_{r} }{(\mu_{s} - \mu_{r})^{2}} \, 
		\etau_{\uu_k}^{2} \eta_{\uu_l}^{2} 
		&=& 
		\E S_{1} \E S_{2},
	\end{EQA}
	where 
	\begin{EQA}
		S_{1} 
		&\eqdef & 
		\mu_{r}^{2} \sum_{k \in \Delta_{r}}  \etau_{\uu_k}^{2}  \etau_{\uu_{k_{1}}}^{2}, \\
		S_{2} 
		&\eqdef & 
		4\sum_{s \neq r} \sum_{l \in \Delta_{s}} \frac{ \mu_{s} \mu_{s_{1}} }{(\mu_{s} - \mu_{r})^{2}(\mu_{s_{1}} - \mu_{r})^{2}}  \eta_{\uu_{l}}^{2} \eta_{\uu_{l_{1}}}^{2}.
	\end{EQA}
	We get that \( \E S_{1} \asymp \mu_{r}^{2} m_{r} \) and 
	\begin{EQA}
		\E S_{2} 
		&\asymp &
		\sum_{s \neq r} \sum_{l \in \Delta_{s}} \frac{ \mu_{s} \mu_{s_{1}} }{(\mu_{s} - \mu_{r})^{2}(\mu_{s_{1}} - \mu_{r})^{2}}.
	\end{EQA}
	Hence,
	\begin{EQA}
		\tilde \sigma^{2} 
		&\asymp &
		n \|\Gamma_{r}\| \Tr \Gamma_{r} = n \|\Gamma_{r}\|^{2} \rr(\Gamma_{r}) .
	\end{EQA}
	Let us introduce the following quantity 
	\begin{EQA}
		\Delta_{1}^{*} 
		&\eqdefapp &
		\|\Gamma_{r}\| \rr^{\frac{1}{2}}(\Gamma_{r}) \left(\sqrt{\frac{\log n}{n}} + \sqrt{\frac{\log p}{n}} \right) 
	\end{EQA}
	and the event \(\Eset_{2} \eqdef \{\|\Gammab_{r}  - \Gamma_{r} \| \leq \Delta_{1}^{*} \}\). Applying Lemma~\ref{matrix bernstein}  we get
	\begin{EQA}
		\Pb(\Eset_{2}^{c}) \leq \Pb( \Eset_{2}^{c} \cap \Eset_{1}) + \frac{1}{n} 
		&\leq &
		\frac{2}{n}.
	\end{EQA}
	To apply Corollary~\ref{cor 2 gauss} we also need to show that the remaining part of the trace of \( \Gammab_{r}\) concentrates around its non-random counterpart. We take \( \mm \) and \( \Pi_\mm\) as stated in the lemma. Denote \( \overline \Pi_\mm \eqdef \II - \Pi_\mm \) and
	\begin{EQA}
		\Eset_3 \eqdef \bigg\{ \bigg| \Tr  \overline\Pi_\mm \Gammab_{r} - \Tr  \overline\Pi_\mm \Gamma_{r} \bigg| 
		&\lesssim &
		\Tr  \overline\Pi_\mm \Gamma_{r} \frac{\log^3 n }{\sqrt n}\bigg\}.
	\end{EQA}
	It is easy to check that
	\begin{EQA}
		\Tr  \overline \Pi_\mm \Gammab_{r} 
		&=& 
		\frac{1}{n}\sum_{j=1}^n \Tr (\Pi Z_{rj}) (\Pi Z_{rj})^{\t} 
		= \frac{2}{n}\sum_{j=1}^n 
		\sum_{s \in \Tset_{r}} \sum_{k \in \Delta_{r}} \sum_{l \in \Delta_{s}} \frac{\mu_{s} \mu_{r}}{(\mu_{s} - \mu_{r})^{2} }\etau_{\uu_k, j}^{2}  \eta_{\uu_l, j}^{2},
	\end{EQA} 
	where \( \etau_{\uu_k, j}, \eta_{\uu_l, j}, k \in \Delta_{r}, l \in \Delta_{s}, s \neq r\), are i.i.d. \(\ND(0,1)\) r.v. 
	Simple calculations show that
	\begin{EQA}
		\E \Tr  \overline \Pi_\mm \Gammab_{r} 
		&=& 
		\Tr  \overline \Pi_\mm \Gamma_{r}.
	\end{EQA}
	We introduce additional notations. 
	Denote \( \gamma_{sr} \eqdef \frac{\mu_{s} \mu_{r}}{(\mu_{s} - \mu_{r})^{2} } \) and
	\begin{EQA}
		Q_{j} 
		&\eqdef &
		\sum_{s \in \Tset_{r}} \sum_{k \in \Delta_{r}} \sum_{l \in \Delta_{s}} \gamma_{sr} [\etau_{\uu_k, j}^{2}  \eta_{\uu_l, j}^{2} - 1], \quad j = 1, \ldots, n.
	\end{EQA}
	It is obvious that \( Q_{j}\) are i.i.d. r.v. We estimate 
	\begin{EQA}
		\E\bigg| \Tr  \overline\Pi_\mm \Gammab_{r} - \Tr  \overline\Pi_\mm \Gamma_{r} \bigg|^{m} 
		&=& 
		\frac{2^{m}}{n^{m}}\E\bigg|\sum_{j=1}^n Q_{j} \bigg|^{m}.
	\end{EQA}
	Applying Rosenthal's inequality (see e.g.~\cite{Rosenthal1970}) we obtain
	\begin{EQA}
		\E\bigg|\sum_{j=1}^n Q_{j} \bigg|^{m} 
		&\le &
		C^{m} \left( m^{\frac m2} n^{\frac m2} \E^{\frac m2} Q_{1}^{2} + m^{m} n \E | Q_{1} |^{m} \right).
	\end{EQA}
	It remains to estimate \( \E | Q_{1} |^{m}  \).  We may rewrite \(Q_{1}\) as follows
	\begin{EQA}
		Q_{1} &=& 
		A Q_{11} + Q_{11} Q_{12} + m_{r} Q_{12},
	\end{EQA}
	where
	\begin{EQA}
		A 
		&\eqdef &
		\sum_{s \in \Tset_{r}} \sum_{l \in \Delta_{s}} \gamma_{sr}, \\
		Q_{11} 
		&\eqdef &
		\sum_{k \in \Delta_{r}} [\etau_{\uu_k}^{2} - 1], \\
		Q_{12} 
		&\eqdef &
		\sum_{s \in \Tset_{r}} \sum_{l \in \Delta_{s}} \gamma_{sr} [\eta_{\uu_l}^{2} - 1].
	\end{EQA}
	Applying Lemma~\ref{th: Rosenthal for sub-exp from 1 to 2} we estimate each term separately and show that
	\begin{EQA}
		\E |Q_{1}|^{m} 
		&\lesssim &
		C^{m} m^{2m} \Tr^{m}  \overline \Pi_\mm \Gamma_{r}. 
	\end{EQA}
	Hence,
	\begin{EQA}
		\E\bigg| \Tr  \overline\Pi_\mm \Gammab_{r} - \Tr  \overline\Pi_\mm \Gamma_{r} \bigg|^{m} 
		&\le &
		\frac{C^{m} m^{3m}}{ n^{\frac m2}} \Tr^{m}  \overline \Pi_\mm \Gamma_{r}.
	\end{EQA}
	Choosing \( m \asymp \log n\) and applying Markov's inequality we get
	\begin{EQA}
		\Pb \left ( \Eset_3 \right) 
		&\geq &
		1 - \frac{1}{n}.
	\end{EQA}
	Denote now \(\Eset = \Eset_{2} \cap \Eset_{3}\). It follows that \(\Pb(\Eset) \geq 1 - \frac{1}{n}\).  Applying Corollary~\ref{cor 2 gauss} we get that for all \( w \in \Eset \)
	\begin{EQA}
		\sup_{x \geq 0} |\Pb(\|\xi\|_{2}^{2} \geq x) - \Pbb(\|\xib\|_{2}^{2} \geq x)| 
		&\lesssim &
		\Delta_{2}^{*},
	\end{EQA}
	where 
	\begin{EQA}
		\Delta_{2}^{*} 
		&\eqdef &
		\frac{\|\Gamma_{r}\| \, \mm\, \rr^{\frac{1}{2}}(\Gamma_{r})}{\sqrt{\lambda_{1}(\Gamma_{r}) \lambda_{2}(\Gamma_{r})}} \left(\sqrt{\frac{\log n}{n}} + \sqrt{\frac{\log p}{n}} \right) +  \frac{\Tr (\II - \Pi_\mm)\Gamma_{r}}{\sqrt{\lambda_{1}(\Gamma_{r}) \lambda_{2}(\Gamma_{r})}}.
	\end{EQA}
	Comparing \(\Delta_{2}^*\) with \(\err_{4}\) we finish the proof of this lemma.
\end{proof}

\subsection{Proof of the main result} 
This section collects the results of the previous sections and provides a proof of   Theorem~\ref{th: main}.

\begin{proof}[Proof of Theorem~\ref{th: main}]
	Let us fix an event \(\Eset \subset \Omega\) which holds with \(\Pb\) - probability at least \(1 - \frac{1}{n}\). 
	Suppose that for all \(\omega \in \Eset\) the statements of Theorems~\ref{th: GAP real world}, \ref{th: GAP bootstrap world} and Lemma~\ref{l: gaussian comparison, our  covariance matrices} hold. First we show that for all \(x > 0\)
	\begin{EQA}
		\label{eq: sup inequality}
		\bigl|\Pbb(n\|\PB_{r} - \PS_{r}\|_{2}^{2} > x) - \Pb(n\|\PS_{r} - \PP_{r}\|_{2}^{2} > x) \bigr| 
		&\lesssim &
		\err,
	\end{EQA}
	where \(\err\) is defined in~(\refeq{def: err}).
	Applying Theorem~\ref{th: GAP bootstrap world} we may show that   
	\begin{EQA}
		\Pbb(n\|\PB_{r} - \PS_{r}\|_{2}^{2} > x) 
		&\geq &
		\Pbb(\|\xib\|_{2}^{2} \geq x-\err_{3}) -  \frac{1}{n}, 
	\end{EQA}
	where we recall that
	\begin{EQA}
		\label{first term}
		\err_{3}
		&\asymp &
		m_{r} \frac{\Tr^3 \CM}{\gu_{r}^{3}} \sqrt{\frac{\log^{3} n}{n} + \frac{\log^{3} p}{n}}.
	\end{EQA}
	Lemma~\ref{l: gaussian comparison, our  covariance matrices} implies	
	\begin{EQA}
		\label{mixed term}
		\Pbb(n\|\PB_{r} - \PS_{r}\|_{2}^{2} > x) 
		&\geq &
		\Pb(\|\xi\|_{2}^{2} \geq x - \err_{3}) - \err_{4} - \frac{1}{n},
	\end{EQA}
	where 
	\begin{EQA}
		\label{second term}
		\qquad \err_{4} 
		&\asymp &
		\frac{\|\Gamma_{r}\| \, \mm\, \rr^{\frac{1}{2}}(\Gamma_{r})}{\sqrt{\lambda_{1}(\Gamma_{r}) \lambda_{2}(\Gamma_{r})}} \left(\sqrt{\frac{\log n}{n}} + \sqrt{\frac{\log p}{n}} \right) +  \frac{\Tr (\II - \Pi_\mm)\Gamma_{r}}{\sqrt{\lambda_{1}(\Gamma_{r}) \lambda_{2}(\Gamma_{r})}}.
	\end{EQA}
	As it is clear from~(\refeq{mixed term}) we need to get bounds for \(\err_{3}\)-band of squared norm of the Gaussian element \(\xi\). 
	For this purpose one can use Lemma~\ref{band of GE}. 
	Then we get from~(\refeq{mixed term})
	\begin{EQA}
		\Pbb(n\|\PB_{r} - \PS_{r}\|_{2}^{2} > x) 
		&\geq &
		\Pb(\|\xi\|_{2}^{2} \geq x) - \overline \err_{3} - \err_{4},
	\end{EQA}	
	where 
	\begin{EQA}
		\label{first term 2}
		\overline \err_{3}
		&\asymp &
		\frac{m_{r} \Tr^3 \CM}{\gu_{r}^{3} \sqrt{\lambda_{1}(\Gamma_{r}) \lambda_{2}(\Gamma_{r})}} \sqrt{\frac{\log^{3} n}{n} + \frac{\log^{3} p}{n}}.
	\end{EQA}
	Finally, applying Theorem~\ref{th: GAP real world} and Lemma~\ref{band of GE} we get 
	\begin{EQA}
		\Pbb(n\|\PB_{r} - \PS_{r}\|_{2}^{2} > x) 
		&\geq &
		\Pb(n\|\PS_{r} - \PP_{r}\|_{2}^{2} > x) 
		- \err_{1}   - \err_{2} - \overline \err_{3} - \err_{4},
	\end{EQA}
	where
	\begin{EQA}
		\err_{1} 
		&\asymp &
		\frac{ m_{r}^{1/2} \Tr \Gamma_{r}}{\sqrt{\lambda_{1}(\Gamma_{r}) \lambda_{2}(\Gamma_{r})}} \left( \sqrt{\frac{\log n}{n}} + \sqrt{\frac{\log p}{n}} \right), 
		\\
		\err_{2} 
		&\asymp &
		\frac{m_{r}  \Tr^3\CM}{\gu_{r}^{3}\sqrt{\lambda_{1}(\Gamma_{r}) \lambda_{2}(\Gamma_{r})}} \sqrt{\frac{\log^3 n}{n}}.
	\end{EQA}
	Similarly we may write down all inequalities in the opposite direction. It is easy to see that
	\begin{EQA}
		\err_{1} + \err_{2} + \overline \err_{3} + \err_{4} 
		&\leq &
		\err,
	\end{EQA}
	where
	\begin{EQA}
		\err
		&\asymp &
		\frac{ \mm\, \Tr \Gamma_{r} }{\sqrt{\lambda_{1}(\Gamma_{r}) \lambda_{2}(\Gamma_{r})}} \left(\sqrt{\frac{\log n}{n}} + \sqrt{\frac{\log p}{n}} \right) +  \frac{\Tr (\II - \Pi_\mm)\Gamma_{r}}{\sqrt{\lambda_{1}(\Gamma_{r}) \lambda_{2}(\Gamma_{r})}} \\
		&&
		+\, \frac{m_{r} \Tr^3\CM}{\gu_{r}^{3} \sqrt{\lambda_{1}(\Gamma_{r}) \lambda_{2}(\Gamma_{r})} } \left(\sqrt{\frac{\log ^{3} n}{n}} 
		+ \sqrt{\frac{\log^{3} p}{n}} \right).
	\end{EQA}
	Hence, we finish the proof of~(\refeq{eq: sup inequality}). Now we show that for all \(w \in \Eset\)  
	\begin{EQA}
		\label{eq: quantile comparison}
		\gamma_{\alpha + \varepsilon_{1}} 
		&\le &
		\gammab_{\alpha} \le \gamma_{\alpha - \varepsilon_{2}}
	\end{EQA}
	with \(\varepsilon_{1} \eqdef 2 \err, \ve_{2} \eqdef \err\). 
	It follows from Theorem~\ref{th: GAP real world}, Lemma~\ref{band of GE} and definition of \(\err\) that
	\begin{EQA}
		\label{eq: quantiles}
		\alpha - \err 
		&\le &
		\Pb(n\|\PS_{r} - \PP_{r}\|_{2}^{2} > \gamma_{\alpha}) \le \alpha.
	\end{EQA}
	The proof of~(\refeq{eq: quantile comparison}) follows from this inequality and~(\refeq{eq: sup inequality}):
	\begin{EQA}
		\Pbb(n\|\PS_{r} - \PP_{r}\|_{2}^{2} > \gamma_{\alpha-\ve_{2}}) 
		&\le &
		\Pb(n\|\PS_{r} - \PP_{r}\|_{2}^{2} > \gamma_{\alpha-\ve_{2}}) + \err  \le \alpha,  \\
		\Pbb(n\|\PS_{r} - \PP_{r}\|_{2}^{2} > \gamma_{\alpha+\ve_{1}}) 
		&\geq &
		\Pb(n\|\PS_{r} - \PP_{r}\|_{2}^{2} > \gamma_{\alpha+\ve_{1}}) - \err \geq  \alpha.
	\end{EQA}
	Hence, applying~(\refeq{eq: quantile comparison}) and~(\refeq{eq: quantiles}) we write
	\begin{EQA}
		\Pb(n\|\PS_{r} - \PP_{r}\|_{2}^{2} > \gammab_{\alpha}) - \alpha 
		&\leq &
		\Pb(n\|\PS_{r} - \PP_{r}\|_{2}^{2} > \gamma_{\alpha+\ve_{1}})- \alpha   \le 2\err, \\
		\Pb(n\|\PS_{r} - \PP_{r}\|_{2}^{2} > \gammab_{\alpha}) - \alpha 
		&\geq &
		\Pb(n\|\PS_{r} - \PP_{r}\|_{2}^{2} > \gamma_{\alpha-\ve_{2}})- \alpha   \geq -2\err.
	\end{EQA}
	The last two inequalities conclude the claim of the theorem.
\end{proof}

\section{Gaussian comparison and anti-concentration inequalities}\label{SgaussianPP}
The aim of this section is to derive dimensional free bound in Gaussian comparison and anti-concentration inequalities. We start from the discussion of the Gaussian comparison inequality.

Due to Pinsker's inequality for any probability measures \(\Pb_{1}\) and \(\Pb_{2}\) on \((\Omega, \mathfrak F)\) we may write
\begin{EQA}
	\label{eq: Pinsker}
	\sup_{A \in \mathfrak{F}} |\Pb_{1}(A) - \Pb_{2}(A)| 
	&\leq &
	\sqrt{\KL(\Pb_{1}, \Pb_{2})/2},
\end{EQA}
where \(\KL(\Pb_{1}, \Pb_{2})\) is the Kullback-Leibler divergence between \(\Pb_{1}, \Pb_{2}\), see~\cite{Tsybakov2009}[pp. 88,132]. 
Let \(\xi\) and \(\eta\) be Gaussian elements in \(\R^{p}\) with zero mean and covariance matrices 
\(\Sigma_{\xi}, \Sigma_{\eta}\) resp.   
Denote \(\W \eqdef \CM_{\xi}^{-\frac{1}{2}} \CM_{\eta} \CM_{\xi}^{-\frac{1}{2}}\) and assume  
\(\|\W - \II\| \leq 1/2\). 
Taking \(\Pb_{1} \eqdef \ND(0, \CM_{\xi})\) and \(\Pb_{2} \eqdef \ND(0, \CM_{\eta})\) one may check (see e.g.~\cite{spokoiny2015} ) that
\begin{EQA}
	\KL(\Pb_{1}, \Pb_{2}) 
	&\leq &
	\frac 12 \Tr (\W - \II)^{2}.
\end{EQA}
The last inequality and~(\refeq{eq: Pinsker}) imply that
\begin{EQA}
	\sup_{z \in \R} |\Pb(\|\xi\| \geq z) - \Pb(\|\eta\| \geq z)| 
	&\leq &
	\frac{1}{2} \|\W - \II\|_{2}.
\end{EQA}
To apply this inequality one have to estimate the Hilbert-Schmidt norm in the r.h.s. of the previous inequality. Below we will show that applying Bernstein's matrix inequality we may control the operator norm \(\|\W - \II\|\). Hence, the r.h.s. of the previous inequality may be bounded up to some constant by \(\sqrt{p}\|\W - \II\|\).  The following lemma shows that in a rather general situation it is possible to derive a  dimensional free bound. We denote by \(\lambda_{1\eta}\geq\lambda_{2\eta} \geq \ldots \) and  
\(\lambda_{1\xi}\geq\lambda_{2\xi} \geq \ldots \)   the eigenvalues of \(\CM_{\eta}\) and \(\CM_{\xi}\) resp. Recall that \( \| \A \|_{1} \) is the Schatten \(1\)-norm (or the trace-class norm), i.e.
\begin{EQA}
	\| \A \|_{1} 
	&\eqdef &
	\Tr |\A |= \sum_{k=1}^\infty |\lambda_k(\A)|.
\end{EQA} 
\begin{lemma}
	\label{l: explicit gaussian comparison}
	Let \(\xi\) and \(\eta\) be Gaussian elements in \(\HM\) with zero mean and covariance operators \(\CM_{\xi}\) and \(\CM_{\eta}\) respectively. The following inequality holds
	\begin{EQA}
		\sup_{x \geq 0} \left|\Pb( \|\xi \|^{2} \geq x) - \Pb( \|\eta \|^{2} \geq x) \right| 
		&\lesssim &
		\left( \frac{1}{\sqrt{\lambda_{1\eta}\lambda_{2\eta}}}
			+ \frac{1}{\sqrt{\lambda_{1\xi}\lambda_{2\xi}}}
		\right) 
		\err_0
	\end{EQA}
	where
	\begin{EQA}
		\err_0 
		&\eqdef &
		\| \CM_\xi - \CM_\eta \|_{1}.
	\end{EQA}	
\end{lemma}

\begin{corollary}
	\label{cor 1 gauss}
	Under assumptions of Lemma~\ref{l: explicit gaussian comparison} the following bound for \( \err_0 \)  holds
	\begin{EQA}
		\label{expl_gauss}
		\err_0 
		&\le & 
		\| \CM_{\xi}^{-\frac{1}{2}} \CM_{\eta} \CM_{\xi}^{-\frac{1}{2}} - \II\| \Tr \CM_{\xi}.
	\end{EQA}
\end{corollary}
\begin{proof}
	The proof follows directly from the following well known inequality
	\begin{EQA}
		\label{Schatten norm inequality}
		\| \A \B\|_{1} 
		&\le &
		\| \A \|_{1} \|\B\|.
	\end{EQA}
\end{proof}

\begin{corollary}
	\label{cor 2 gauss}
	Let \(m: 1 \le m < \infty \) and  \( \Pi_{m} \eqdef \sum_{k=1}^{m} \ee_{j} \ee_{j}^{\t}\). Under assumptions of Lemma~\ref{l: explicit gaussian comparison} the following bound for \( \err_0 \)  holds
	\begin{EQA}
		\label{eq: err0 term}
		\err_0   
		&\le &
		m \| \CM_{\xi} - \CM_{\eta} \| + \Tr (\II-\Pi_{m}) \CM_\xi + \Tr (\II -  \Pi_{m}) \CM_\eta.
	\end{EQA}
\end{corollary}
\begin{proof}
	The proof is obvious.
\end{proof}

\begin{remark} It is easy to see that we may assume without loss of generality that \( \CM_\xi\) and \( \CM_{\eta}\) are diagonal matrices. Then the last two terms in~(\refeq{eq: err0 term}) are the sums of eigenvalues \(\lambda_{j \xi}, \lambda_{j \eta}, j \geq m+1\).
\end{remark}

It next lemma we show that in a rather general situation one may obtain dimensional free anti-concentration inequality for the squared norm of Gaussian element with dependence on first two largest eigenvalues of \(\CM\) only. 

\begin{lemma}[\(\Delta\)-band of the squared norm of a Gaussian element]\label{band of GE}
	Let \(\xi\) be a Gaussian element in \(\HM\) with zero mean and covariance operator  \(\CM\). Then for arbitrary  \(\Delta > 0\) and any \(\lambda > \lambda_{1}\)      
	\begin{EQA}
		\label{band of Gaussian1}
		\Pb(x < \|\xi \|^{2} < x + \Delta)  
		&\leq & 
		C_1 \, \Delta,
	\end{EQA}
	where
	\begin{EQA}
		C_1 
		&\eqdef &
		\frac{e^{-x/(2\,\lambda)}}{\sqrt{\lambda_{1}\lambda_{2}}}
		\prod_{j=3}^{\infty}\,(1-\lambda_{j}/\lambda)^{-1/2}
	\end{EQA}
	and \(\lambda_{1}\geq\lambda_{2}\geq\dots\) are the  eigenvalues of \(\CM\). In particular, one has
	\begin{EQA}
		\label{band of Gaussian2}
		\sup_{x  > 0} \Pb(x < \|\xi \|^{2} < x + \Delta)  
		&\leq & 
		\frac{\Delta}{\sqrt{\lambda_{1}\lambda_{2}}}.
	\end{EQA}	
\end{lemma}

\begin{remark}  The right-hand sides of~(\refeq{band of Gaussian1}) and~(\refeq{band of Gaussian2})
	depend on first two eigenvalues of \(\CM\). In general it is impossible to get similar bounds of order \(O( \Delta)\) with dependence on \(\lambda_{1}\) only. It is easy to get in one dimensional case, i.e.  when \(\lambda_{1} = 1\) and  \(\lambda_{2} = 0\), that for all positive \(\Delta \leq \log 2\) one has
	\begin{EQA}
		\sup_{x  > 0} \Pb(x < \|\xi \|^{2} < x + \Delta)  
		&\geq &
		\Delta^{1/2}/(2\sqrt{\pi}).
	\end{EQA}		
\end{remark}

\begin{proof}[Proof of Lemma~\ref{l: explicit gaussian comparison}] 
	Fix any \(s: 0\leq s \leq 1\). Let  \(Z(s)\) be a Gaussian random element in \(\HM\) with zero mean and covariance operator \(\V(s)\):
	\begin{EQA}
		\V(s) 
		&\eqdef & 
		s \CM_{\xi} + (1-s) \CM_{\eta}.
	\end{EQA}
	Denote by \(\lambda_{1}(s)\geq\lambda_{2}(s)\geq\dots\) 
	the eigenvalues of \(\V(s)\). 
	It is straightforward to check that a characteristic function \(f(t,s)\) of \(\|Z(s) \|_{2}^{2}\) can be written as
	\begin{EQA}
		\label{eq:  c.f.}
		f(t,s) 
		&=& 
		\E \exp\{it\|Z(s) \|^{2}\} = \prod_{j=1}^{\infty}(1-2it\lambda_{j}(s))^{-1/2}
		\\
		&=&
		\exp \left\{ - \frac{1}{2} \sum_{j=1}^{\infty} \log (1-2it\lambda_{j}(s)) \right\}. 
	\end{EQA}
	Indeed, one may use the following representation
	\begin{EQA}
		\label{eq: gaussian representation}
		Z(s) 
		&=& 
		\sum_{j=1}^{\infty} \sqrt{\lambda_{j}(s)} \, Z_{j} \, \ee_{j}, 	
	\end{EQA}	
	where \(Z_{j}, j \geq 1,\) are i.i.d. \(\ND(0,1)\) r.v. and \(\ee_{j}, j \geq 1,\) be an orthonormal basis in \(\HM\). Then it is sufficient to apply an  expression for a characteristic function of \( Z_{j}^{2} \).  We rewrite \(f(t,s)\) in terms of trace-class operators
	\begin{EQA}
		\label{char funct}
		f(t,s) 
		&=& 
		\exp\bigl\{ - \frac{1}{2} \Tr \log\bigl( \II - 2 it \V(s) \bigr) \bigr\},
	\end{EQA}
	where for an operator \(\A\) and the identity operator \(\II\) we use notation  
	\begin{EQA}
		\log(\II + \A) 
		&=& 
		\A \int_{0}^{1}(\II + y \A)^{-1}dy.
	\end{EQA}
	It is well known, see e.g.~\cite{Chung2001}[\S 6.2, p. 168], that
	for a continues d.f. \(F(x)\) with c.f. \(f(t)\) we may write
	\begin{EQA}
		F(x) 
		&=& 
		\frac{1}{2} + \frac{i}{2\pi} \lim_{T \rightarrow \infty} \text{V.P.} \int_{|t| \le T} e^{-i t x} f(t) \frac{dt}{t}.
	\end{EQA} 
	Let us fix an arbitrary \(x > 0\). Then 
	\begin{EQA} 
		\Pb( \|\xi \|^{2} < x) - \Pb( \|\eta \|^{2} < x) 
		&=&  
		\frac{i}{2\pi} \lim_{T \rightarrow \infty} \text{V.P.} 
			\int_{|t| \le T} \frac{f(t,1) - f(t,0)}{t} e^{-i t x} \,dt.
	\end{EQA}
	Since
	\begin{EQA}
		f(t,1) - f(t,0)
		&=&
		\int_{0}^{1}\frac{\partial f(t,s)}{\partial s}\,ds,
	\end{EQA}
	changing the order of integration we get
	\begin{EQA}
		\label{smooth inequal}
		\Pb( \|\xi \|^{2} < x) &-& \Pb( \|\eta \|^{2} < x)
		 \\
		&= &
		\frac{i}{2\pi} \lim_{T \rightarrow \infty} \text{V.P.} \int_0^1\int_{|t| \le T} \frac{\partial f(t,s)/\partial s}{t} e^{-i t x} \, dt \, ds
	\end{EQA}
	It is easy to check that
	\begin{EQA}
		\frac{\partial f(t,s)/\partial s}{t}\ 
		&=& 
		f(t,s) \, \Tr\bigl\{ (\CM_{\xi} - \CM_{\eta}) 
		\bigl( \II - 2 it \V(s) \bigr)^{-1} \bigr\}  \\  
		&=& 
		f(t,s) \, \Tr\bigl\{ (\CM_{\xi} - \CM_{\eta}) \G(t,s) \bigr\},  
	\end{EQA}
	where \(\G(t,s) \eqdef (\II - 2it\, \V(s))^{-1} \).  
	Hence,
	\begin{EQA}
		\Pb( \|\xi \|^{2} < x) - \Pb( \|\eta \|^{2} < x)
		&=& 
		\frac{i}{2\pi} \lim_{T \rightarrow \infty}  \int_0^1 \Tr\left\{ (\CM_{\xi} - \CM_{\eta}) \Gh(T, s) \right\} \, ds,
	\end{EQA}
	where 
	\begin{EQA}
		\Gh(T, s) 
		&\eqdef &
		\int_{|t| \le T} f(t,s)  \G(t,s) e^{-i t x} \, dt.
	\end{EQA}
	We show that for any \(T>0\) and \(s\in [0, 1] \) one has 
	\begin{EQA}
		\label{bound_for_G}
		\|\Gh(T,s)\| 
		&\le &
		\frac{c}{\sqrt{\lambda_{1}(s) \lambda_{2}(s)}}.
	\end{EQA}
	For this aim we denote the eigenvalues of \(\G(t,s)\) by \(\mu_{j}(t,s) \eqdef (1 - 2 i t \lambda_{j}(s))^{-1}\). Let \( \overline Z_{j} \) be a random variable with exponential distribution \(Exp(0, 1/(2\lambda_{j}(s)))\), which is independent of \(Z_{k}, k \geq 1\). Then
	\begin{EQA}
		\label{c.f. Z}
		\E e^{i t \overline Z_{j}} 
		&=& 
		\mu_{j}(t,s).
	\end{EQA}
	Applying~(\refeq{c.f. Z}) we obtain 
	\begin{EQA}
		\label{eq: product c.f}
		f(t,s) \mu_{j}(t,s) 
		&=& 
		\E \exp \bigg( it \big[\sum_{k \geq 1}\lambda_{k}(s) Z_{k}^{2} + \overline Z_{j}\big]\bigg) \\
		&=& 
		\E e^{it a_{j}^{2}} \cdot  \E \bigg(\exp \big( it \big [\lambda_{1}(s) Z_{1}^{2} + \lambda_{2}(s) Z_{2}^{2}\big]\big)  \bigg ),
	\end{EQA}
	where \(a_{j}^{2} \eqdef  \overline Z_{j} + \sum_{k \geq 3}\lambda_{k}(s) Z_{k}^{2}\). 
	We fix \(j\) and get a bound for 
	\begin{EQA}
		I &\eqdef &
		\left| \int_{-T}^{T} f(t,s) \mu_{j}(t,s) e^{-i t x} \, dt \right|.
	\end{EQA}
	Applying~(\refeq{eq: product c.f}) we obtain
	\begin{EQA}
		I \le  \E\left|  \int_{-T}^{T}  e^{i t (a_{j}^{2} - x)} \E \exp \bigg( it \big [\lambda_{1}(s) Z_{1}^{2} + \lambda_{2}(s) Z_{2}^{2}\big]\bigg)   \, dt \right|.
	\end{EQA}
	It follows from~\cite{GotzUlyanov2000}[Lemma 2.2] (see also~\cite{ProkhUlyanov2013}[p. 242]) that there exists an absolute constant \(c\) such that
	\begin{EQA}
		\qquad \left|  
		\int_{-T}^{T}  e^{i t (a_{j}^{2} - x)} 
		\E \exp \bigg( it \big [\lambda_{1}(s) Z_{1}^{2} + \lambda_{2}(s) Z_{2}^{2}\big]\bigg) dt 
		\right| 
		&\le &
		\frac{c}{\sqrt{\lambda_{1}(s) \lambda_{2}(s)}}. \label{eq: conj}
	\end{EQA}
	For readers convenience we repeat the proof of this inequality below in Lemma~\ref{l: aux lemma int est}. Applying~(\refeq{eq: conj}) we get that the absolute values of all eigenvalues of \(\Gh(T, s)\) are bounded by \(c (\lambda_{1}(s) \lambda_{2}(s))^{-1/2}\) and, therefore, we obtain~(\refeq{bound_for_G})
	\begin{EQA}
		\left|\Tr\left\{ (\CM_{\xi} - \CM_{\eta}) \Gh(T, s) \right\} \right| 
		&\le &
		\| \CM_\xi - \CM_\eta\|_{1} \|\Gh(T,s) \| 
		\le 
		\frac{c \| \CM_\xi - \CM_\eta\|_{1} }{\sqrt{\lambda_{1}(s) \lambda_{2}(s)}}.
	\end{EQA}
	This implies the claim of the lemma.  
\end{proof}	

\begin{proof}[Proof of Lemma~\ref{band of GE}] 
	The inequality~(\refeq{band of Gaussian2}) follows immediately from~(\refeq{band of Gaussian1}) if we take \(\lambda = 2 \Tr \CM\) and note that
	\begin{EQA}
		\prod_{j=3}^{\infty}\,(1-\lambda_{j}/\lambda)^{-1/2} 
		&\leq &
		\left(1 - \lambda^{-1}\sum_{j=3}^{\infty}\lambda_{j} \right)^{-1/2} 
		\leq \sqrt{2}.
	\end{EQA}
	In order to prove~(\refeq{band of Gaussian1}) it is sufficient to show that for a density function \(g(u)\) of \(\|\xi \|_{2}^{2}\) one has
	\begin{EQA}
		\label{gaussNorm0}
		g(u) 
		&\leq &
		\frac{ e^{-u/(2\,\lambda)}}{\sqrt{\lambda_{1}\lambda_{2}}}
		\prod_{j=3}^{\infty}\,(1-\lambda_{j}/\lambda)^{-1/2}.
	\end{EQA}
	According to representation~(\refeq{eq: gaussian representation})	
	\begin{EQA}
		\|\xi \|^{2} &=& 
		\sum_{j=1}^{\infty} \lambda_{j}\,Z_{j}^{2},
	\end{EQA}
	where \(Z_{1}, Z_{2}, \dots\) are i.i.d. \(\ND(0,1)\) r.v. We denote by \(g(m,u), \, m=1,2,\dots\) (resp. \(g_{j}(u), \, j=1,2,\dots\)) the density function of 
	\(\sum_{j=1}^{m} \lambda_{j}\,Z_{j}^{2}\) (resp. \(\lambda_{j}\,Z_{j}^{2}\)). We have for all \(j=1,2,\dots\) and any \(\lambda > \lambda_{1}\) 
	\begin{EQA}
	\label{gaussNorm1}
	g_{j}(u) 
	&=& 
	(2\pi u\lambda_{j})^{-1/2} d_{j}(u) \\
	&\leq&
	(2\pi u\lambda_{j})^{-1/2} \exp\{-u/(2\lambda_{1})\} d(\lambda\lambda_{j}/(\lambda - \lambda_{j}), u), 
	\end{EQA}
	where \(d_{j}(u) = d(\lambda_{j}, u) = \exp\{-u/(2\lambda_{j})\}\). 
	Moreover,
	\begin{EQA}
		&&
		(2\pi u)^{-1/2}(\lambda - \lambda_{j})^{1/2}/(\lambda\lambda_{j})\,d(\lambda\lambda_{j}/(\lambda - \lambda_{j}), u)
	\end{EQA}
	is the density function of \(Z_{j}^{2}\sqrt{\lambda}/(\lambda - \lambda_{j})^{1/2}\). 
	First consider \(g(2,u)\):
	\begin{EQA}
		\label{gaussNorm2}
		g(2,u) 
		&=& 
		\int_{0}^{u} g_{1}(u-v)g_{2}(v)dv \\
		& \leq &
		\frac{\exp\{-u/(2\lambda_{1})\}}{2\pi\sqrt{\lambda_{1}\lambda_{2}}}\int_{0}^{1}\frac{dz}{\sqrt{(1-z)z}}  = 
		\frac{\exp\{-u/(2\lambda_{1})\}}{2\sqrt{\lambda_{1}\lambda_{2}}}.
	\end{EQA}
	Therefore, due to~(\refeq{gaussNorm1}) and~(\refeq{gaussNorm2}) we obtain
	\begin{EQA} 
		g(3,u) 
		&=& 
		\int_{0}^{u} g_{2}(u-v)g_{3}(v)dv \\
		& \leq &
		\frac{\exp\{-u/(2\lambda_{1})\}}{2\sqrt{\lambda_{1}\lambda_{2}}\,
			\sqrt{2\pi \lambda_{3}}}\int_{0}^{u}\frac{d(\lambda\lambda_{3}/(\lambda - \lambda_{3}), v)}{\sqrt{v}}dv 
		\\
		& \leq &
		\frac{\exp\{-u/(2\lambda)\}}{2\sqrt{\lambda_{1}\lambda_{2}}}(1 - \lambda_{3}/\lambda)^{-1/2}.
	\end{EQA}
	In a similar way by induction we can get for any \(m>3\) that
	\begin{EQA}
		\label{gaussNorm3}
		g(m,u) 
		&\leq &  
		\frac{e^{-u/(2\,\lambda)}}{2\sqrt{\lambda_{1}\lambda_{2}}}
		\prod_{j=3}^{m}\,(1-\lambda_{j}/\lambda)^{-1/2}.
	\end{EQA}
	Now take an arbitrary \(\ve > 0\) and any integer \(m>0\). Let \(0 < \mu < 1/(2\lambda_{j})\) for all \(j \geq m+1\). Without loss of generality we assume that at least two \(\lambda_{j}, j \geq m+1\), are non-zero. Otherwise the arguments are simpler. By Markov's inequality  we obtain
	\begin{EQA}
		\Pb\left(\sum_{j=m+1}^{\infty}\lambda_{j} Z_{j}^{2}\ge \ve^{2}\right)
		&\le &
		e^{-\mu \ve^{2}} \prod_{j=m+1}^{\infty} \E e^{\mu \lambda_{j} Z_{j}^{2}} 
		= e^{-\mu \ve^{2}} \prod_{j=m+1}^{\infty} \frac{1}{\sqrt{1-2\mu \lambda_{j}}}.
	\end{EQA}
	Choosing \(\mu \eqdef 1/(2\sum_{j=m+1}^{\infty} \lambda_{j})\) we get
	\begin{EQA}
		\Pb\left(\sum_{j=m+1}^{\infty}\lambda_{j} Z_{j}^{2}\ge \ve^{2}\right)
		&\le &
		2\exp\left\{-\ve^{2} \left(2 \sum_{j=m+1}^{\infty}\lambda_{j}\right)^{-1} \right\}.
	\end{EQA}
	Hence, there exists \(M = M(\ve)\) such that for all \(m\ge M\)
	\begin{EQA}
		\Pb\left(\sum_{j=m+1}^{\infty}\lambda_{j} Z_{j}^{2}\ge \ve^{2}\right)
		&\leq &
		\ve^{2}.
	\end{EQA}
	Therefore, for any \(m\ge M\) 
	\begin{EQA}
		\label{gaussNorm4}
		\Pb(x - \ve < \|\xi \|_{2}^{2} < x + \ve) 
		&\leq &
		\ve^{2} + 2(\ve + \ve^{2})\sup_{y\in T(\ve, x)} g(m, y),
	\end{EQA}
	where \(T(\ve, x) = \{y \in \mathbb{R}^{1}: x - \ve -\ve^{2} \leq y \leq x + \ve + \ve^{2}\}\). 
	Dividing the right-hand side of~(\refeq{gaussNorm4}) by \(\ve\) we obtain~(\refeq{gaussNorm0}) from~(\refeq{gaussNorm3}) as \(\ve\) tends to 0.	
\end{proof}

\appendix

\section{Auxiliary results}
\label{SappendPP}
\subsection{Concentration inequalities for sample covariances and spectral projectors in \(\X\) - world}
In this section we present concentration inequalities for sample covariance matrices and spectral projectors in \(\X\) - world.
\begin{theorem}
	\label{th: concentration for sample cov real world}
	Let \(X, X_{1}, \ldots , X_{n}\) be i.i.d. centered Gaussian random vectors in \(\R^{p}\) with covariance \(\CM = \E(X X^{\T})\). Then
	\begin{EQA}
		\E \|\SC - \CM\| 
		&\lesssim &
		\|\CM\| \left(\sqrt{\frac{\rr(\CM)}{n}} + \frac{\rr(\CM)}{n}\right).
	\end{EQA}
	Moreover, for all \(t \geq 1\) with probability \(1 - e^{-t}\)
	\begin{EQA}
		\|\SC - \CM\| 
		&\lesssim &
		\|\CM\| \left[ \sqrt{\frac{\rr(\CM)}{n}} \bigvee \frac{\rr(\CM)}{n} \bigvee \sqrt{\frac{t}{n}} \bigvee \frac{t}{n} \right].
	\end{EQA}
\end{theorem}
\begin{proof}
	See~\cite{KoltchLounici2015b}[Theorem 6, Corollary 2].
\end{proof}

To deal with spectral projectors we need the following result which was proved in~\cite{KoltchLounici2015}. Let us introduce additional notations. We denote by \(\CMt\) an arbitrary perturbation of \(\CM\) and \(\EEt \eqdef \CMt - \CM\). Recall that
\begin{EQA}
	\CC_{r} 
	&\eqdef &
	\sum_{s \neq r} \frac{1}{\mu_{r} - \mu_{s}} \PP_{s}.
\end{EQA}

\begin{lemma}\label{decomposition real world}
	Let \(\CMt\) be an arbitrary perturbation of \(\CM\) and let \( \PPt_{r} \) be the corresponding projector. The following bound holds:
	\begin{EQA}
		\label{eq: decomposition real world 1}
		\| \PPt_{r} - \PP_{r}\|
		&\leq &
		4 \frac{\|\EEt\|}{\gu_{r}}.
	\end{EQA}
	Moreover, \(\PPt_{r} - \PP_{r} = L_{r}(\EEt) + S_{r}(\EEt)\),
	where \(L_{r}(\EEt) \eqdef \CC_{r} \EEt \PP_{r} + \PP_{r} \EEt \CC_{r}\)
	and 
	\begin{EQA}
		\label{eq: decomposition real world 2}
		\|S_{r}(\EEt)\| 
		&\leq &
		14 \left(\frac{\|\EEt\|}{\gu_{r}} \right)^{2}.
	\end{EQA}
\end{lemma}
\begin{proof}
	See~\cite{KoltchLounici2015}[Lemma~1].
\end{proof}

\begin{theorem}[Concentration results in \(\X\) - world]\label{th: concentration real world}
	Assume that the conditions of Theorem~\ref{th: main} hold. Then for all \(t: 1 \leq t \leq n^{1/4}\) and
	\begin{EQA}
		\label{eq: real world cond}
		\frac{\Tr \CM}{\gu_{r}} \left(\sqrt{\frac{t}{n}} + \sqrt{\frac{\log p}{n}} \right) 
		&\lesssim &
		1,
	\end{EQA}
	the following bound holds with probability at least \(1 - e^{-t}\)
	\begin{EQA}
		\label{eq: bound 3}
		\Bigl|\|\PS_{r} - \PP_{r}\|_{2}^{2} - \|L_{r}(\EE)\|_{2}^{2} \Bigr| 
		&\lesssim &
		m_{r} \frac{\|\CM\|^{3} \rr^{3}(\CM)}{\gu_{r}^{3}} \left( \frac{t}{n}\right)^{3/2}.
	\end{EQA}
\end{theorem}
\begin{proof}
	The proof follows from~\cite{KoltchLounici2015}[Theorem~3,~5].
\end{proof}

\subsection{Concentration inequalities for sums of random variables and random matrices}
In what follows for a vector \(a = (a_{1}, \ldots, a_{n})\) we denote \(\|a\|_{s} \eqdef \big(\sum_{k=1}^{n} |a_{k}|^{s}\big)^{1/s}\). For a random variable \(X\) and \(r > 0\) we define the \(\psi_{r}\)-norm by
\begin{EQA}
	\|X\|_{\psi_{r}} 
	&\eqdef &
	\inf \{ C > 0: \E \exp(X/C)^{r} \leq 2\}. 
\end{EQA}
If a random variable \(X\) is such that for any \(p \geq 1, \E^{1/p} |X|^{p} \leq p^{1/r} K\), for some \(K > 0\), then \(\|X\|_{\psi_{r}} \leq c K\)
where \(c > 0\) is a numerical constant.

\begin{lemma}\label{th: Rosenthal for sub-exp from 1 to 2}
	Let \(X, X_{i}, i = 1, \ldots, n\) be i.i.d. random variables with \(\E X = 0\) and \(\|X\|_{\psi_{r}} \leq 1, 1 \leq r \leq 2\). Then
	there exists some absolute constant \(C>0\) such that for all \(p \geq 1\)
	\begin{EQA}
		\E\left| \sum_{k=1}^{n} a_{k } X_{k} \right|^{p} 
		&\leq &
		(C p)^{p/2} \|a\|_{2}^{p} +  (C p)^{p} \|a\|_{r_{*}}^{p},
	\end{EQA}
	where \(a = (a_{1}, \ldots, a_{n})\) and \(1/r + 1/r_{*} = 1\). 
\end{lemma}
\begin{proof}
	See~\cite{AdamczakLitvak2011}[Lemma~3.6].
\end{proof}

\begin{lemma} \label{th: moment inequality for sub-exp from 0 to 1}
	If \(0 < s < 1\) and \(X_{1}, . . ., X_{n}\) are independent random variables satisfying \(\|X\|_{\psi_{s}} \leq 1\), then for all \(a = (a_{1}, \ldots , a_{n}) \in \R^{n}\) and \(p \geq 2\)   
	\begin{EQA}
		\E\left|\sum_{k=1}^{n} a_{k} X_{k}\right|^{p} 
		&\leq &
		(C p)^{p/2}  \|a\|_{2}^{p} + C_{s} p^{p/s} \|a\|_{p}^{p}.
	\end{EQA}
	Moreover, for \(s \geq 1/2\), \(C_{s}\) is bounded by some absolute constant.
\end{lemma}
\begin{proof}
	See~\cite{AdamczakLitvak2011}[Lemma~3.7].
\end{proof}

\begin{lemma}\label{sum of 4th powers}
	Let \(\eta_{1}, \ldots , \eta_{n}\) be i.i.d. standard normal random variables. For all \(t\geq 1\) 
	\begin{EQA}
		\label{eq: 4 powers}
		\Pb\left(\left|\sum_{i=1}^{n} a_{i}(\eta_{i}^{4} - 3)\right| \gtrsim t^{2} \|a\|_{2} 
		\right) 
		&\leq &
		e^{-t}.
	\end{EQA}
	Moreover, if \(\etau_{1}, \ldots , \etau_{n}\) are i.i.d. standard normal random variables and independent of  \(\eta_{1}, \ldots , \eta_{n}\) then
	\begin{EQA}
		\label{eq: 2 times 2 powers}
		\Pb\left(\left|\sum_{i=1}^{n} a_{i}(\eta_{i}^{2} \etau_{i}^{2} - 1)\right| 
		\gtrsim 
		t^{2} \|a\|_{2} \right) 
		&\leq &
		e^{-t}.
	\end{EQA}
	
\end{lemma}
\begin{proof}
	We prove~(\refeq{eq: 2 times 2 powers}) only. The proof of~(\refeq{eq: 4 powers}) is similar. Let \(\epsilon_{i}, i = 1, \ldots, n\), be i.i.d. Rademacher r.v. Denote \(\xi_{i}  \eqdef \eta_{i}^{2} \etau_{i}^{2} - 1, i = 1, \ldots, n\). Applying Lemma~\ref{th: moment inequality for sub-exp from 0 to 1} with \(s = 1/2\) we write
	\begin{EQA}
		\E |\sum_{i=1}^{n} a_{i}\xi_{i}|^{p} 
		&\leq &
		2^{p} \E |\sum_{i=1}^{n} a_{i} \epsilon_{i} \xi_{i}|^{p} 
		\leq  
		C^{p} p^{p/2} \|a\|_{2}^{p} + C^{p} p^{2p} \|a\|_{p}^{p}  \leq C^{p} p^{2p} \|a\|_{2}^{p}.
	\end{EQA}
	From Markov's inequality 
	\begin{EQA}
		\Pb\left(\left|\sum_{i=1}^{n} a_{i}(\eta_{i}^{2} \etau_{i}^{2} - 1)\right| \geq t^{2} \|a\|_{2} \right) 
		&\leq &
		\frac{C^{p} p^{2p}}{t^{2 p}}.
	\end{EQA}
	Taking \(p = t/(Ce)^{1/2}\) we finish the proof of the lemma.
\end{proof}

\begin{lemma}[Matrix Gaussian Series]
	\label{matrix series}
	Consider a finite sequence \(\{\A_{k}\}\) of fixed, self-adjoint matrices with dimension \(d\), and let \(\{\xi_{k}\}\) be a finite sequence
	of independent standard normal random variables. Compute the variance parameter
	\begin{EQA}
		\sigma^{2} 
		&\eqdef &
		\left\|\sum_{k=1}^{n} \A_{k}^{2} \right\|.
	\end{EQA}
	Then, for all \(t \geq 0\),
	\begin{EQA}
		\Pb\left(\left\|\sum_{k=1}^{n} \xi_{k} \A_{k} \right\| \geq t\right) 
		&\leq & 
		2d \,\exp(- t^{2}/2\sigma^{2}).
	\end{EQA}
\end{lemma}
\begin{proof}
	See in~\cite{Tropp2012}[Theorem~4.1].
\end{proof}

\begin{lemma}[Matrix Bernstein inequality]
	\label{matrix bernstein}
	Consider a finite sequence \({\X_{k}}\) of independent, random, self-adjoint matrices with dimension \(d\). Assume that
	\(\E \X_{k} = 0\) and \(\lambda_{\max}(\X_{k}) \leq R\) almost surely. Compute the norm of the total variance,
	\begin{EQA}
		\sigma^{2} 
		&\eqdef &
		\left\|\sum_{k=1}^{n} \E \X_{k}^{2} \right\|.
	\end{EQA}
	Then the following inequalities hold for all \(t \geq 0\):
	\begin{EQA}
		\Pb\left( \lambda_{\max}\left(\sum_{k=1}^{n} \X_{k} \right) \geq t\right) 
		&\leq &
		d\, \exp\left(- \frac{t^{2}/2}{\sigma^{2} + R t/3}\right).
	\end{EQA}
	Moreover, if \(\E \X_{k} = 0\) and \(\E \X_{k}^{p} \preceq \frac{p!}{2} R^{p-2} \A_{k}^{2}\) then 
	the following inequalities hold for all \(t \geq 0\):
	\begin{EQA}
		\Pb\left(\lambda_{\max}\left(\sum_{k=1}^{n} \X_{k} \right) \geq t\right) 
		&\leq &
		d \,\exp\left(- \frac{t^{2}/2}{\tilde{\sigma}^{2} + R t}\right),
	\end{EQA}
	where
	\begin{EQA}
		\tilde{\sigma}^{2} 
		&\eqdef &
		\Big\|\sum_{k=1}^{n} \A_{k}^{2} \Big\|.
	\end{EQA}
\end{lemma}
\begin{proof}
	See in~\cite{Tropp2012}[Theorem~6.1]. 
\end{proof}
\subsection{Auxiliary lemma}
\begin{lemma}\label{l: aux lemma int est} 
Assume that \(Z_1, Z_2\) be i.i.d. and \(\ND(0,1)\). Let \( \lambda_1, \lambda_2 \) be any positive numbers and \(b \neq 0\). There exists an absolute constant \(c\) such that
\begin{EQA}
	\qquad \left|  
		\int_{-T}^{T}  e^{i t b} 
		\E \exp \bigg( it \big [\lambda_{1} Z_{1}^{2} + \lambda_{2} Z_{2}^{2}\big]\bigg) dt 
	\right| 
	&\le &
	\frac{c}{\sqrt{\lambda_{1} \lambda_{2}}}. \label{eq: conj 1}
\end{EQA}
\begin{proof}
Denote the l.h.s. of~(\refeq{eq: conj 1})  by \(I'\). Using Euler's formula for complex exponential function we get for positive \(g\) and any \(d \in \R \) 
\begin{EQA}
	g + i d 
	&=& 
	\sqrt{g^{2} + d^{2}} e^{i \zeta}, \quad \zeta = \arcsin \frac{d}{\sqrt{g^{2} + d^{2}}}.
\end{EQA}
Hence, by~(\refeq{eq:  c.f.}) we get
\begin{EQA}
	I' &=& 
	\left| \int_{-T}^{T}  \exp\bigg(i t b + \sum_{k=1}^{2} \frac{i \phi_{k} }{2}\bigg) \prod_{k=1}^{2} 
		\left(1+4t^{2}\lambda^{2}_{k}\right)^{-1/4} 
	\right|,
\end{EQA}
where \(\phi_{k} \eqdef \phi_{k}(t) \eqdef \arcsin\big(2 \lambda_{k} t/(1 + 4 t^{2} \lambda_{k}^{2})^{\frac{1}{2}}\big)\). 
Since \(\prod_{k=1}^{2}\left(1+4t^{2}\lambda^{2}_{k}\right)^{-1/4}\) is even function and \(\phi_{k}(t), k = 1,2\), is odd function of \(t\), we may rewrite \(I'\) as follows
\begin{EQA}
	I'
	&=& 
	\frac{2}{\sqrt{\lambda_{1} \lambda_{2}}}
	\left| \int_{0}^{T} \frac{1}{t} 
	\sin\bigg(t b + \sum_{k=1}^{2} \frac{1}{2} \bigg( \phi_{k} - \frac \pi 2 \bigg) \bigg) 
	\prod_{k=1}^{2}\left( \frac{t^{2} \lambda_{k}^{2}}{1+4t^{2}\lambda^{2}_{k}}\right)^{1/4} \, dt
	\right|.	
\end{EQA} 
We note that
\begin{EQA}
	\prod_{k=1}^{2}\left( \frac{t^{2} \lambda_{k}^{2}}{1+4t^{2}\lambda^{2}_{k}}\right)^{1/4} 
	&\le &
	\sqrt{|t| \lambda_{2}}	
\end{EQA}
Hence, to prove~(\refeq{eq: conj 1}) it is enough to show that
\begin{EQA}
	\label{eq: int 0}
	I'' 
	&\eqdef &
	\left| \int_{1/\lambda_{2}}^{T} 
	\frac{1}{t} \sin\bigg(t b + \sum_{k=1}^{2} \frac{1}{2} \bigg( \phi_{k} - \frac{\pi}{2} \bigg) \bigg) 
	\prod_{k=1}^{2}\left( \frac{t^{2} \lambda_{k}^{2}}{1+4t^{2}\lambda^{2}_{k}}\right)^{1/4} \, dt 
	\right| \le c.
\end{EQA}
We may rewrite \(I''\) as follows 
\begin{EQA}
	I'' 
	&\le & 
	I_{1}'' + \ldots + I_4'',
\end{EQA} 
where
\begin{EQA}
	I_{1}'' 
	&\eqdef & 
	\left| \int_{1/\lambda_{2}}^{T} \frac{1}{t}  \sin(t b)\, dt \right|, \\
	I_{2}'' 
	&\eqdef &
	\left| \int_{1/\lambda_{2}}^{T} \frac{1}{t} \left[\sin\bigg(t b + \sum_{k=1}^{2} \frac{1}{2} \bigg( \phi_{k} - \frac \pi 2 \bigg) \bigg)- \sin(t b) \right]\, dt 
	\right|, \\
	I_{3}'' 
	&\eqdef &
	\left| \int_{1/\lambda_{2}}^{T} \frac{1}{t} \sin\bigg(t b + \sum_{k=1}^{2} \frac{1}{2} \bigg( \phi_{k} - \frac \pi 2 \bigg) \bigg) \left[1 -\left( \frac{t^{2} \lambda_{1}^{2}}{1+4t^{2}\lambda^{2}_{1}}\right)^{1/4}  \right]  \, dt \right|, 
	\\
	I_{4}'' 
	&\eqdef &
	\left| \int_{1/\lambda_{2}}^{T} \frac{1}{t} \sin\bigg(t b + \sum_{k=1}^{2} \frac{1}{2} \bigg( \phi_{k} - \frac \pi 2 \bigg) \bigg) \right. 
	\\
	&& \qquad\qquad\qquad\left. \times \left[1 -\left( \frac{t^{2} \lambda_{2}^{2}}{1+4t^{2}\lambda^{2}_{2}}\right)^{1/4}  \right] \left( \frac{t^{2} \lambda_{1}^{2}}{1+4t^{2}\lambda^{2}_{1}}\right)^{1/4} \, dt \right|.
\end{EQA}
The bound \(I_{1}'' \le c\) is true since for any positive \(A\) and \(B\) we have
\begin{EQA}
	\left| \int_{A}^{B} \frac{\sin t}{t}  \, dt \right| 
	&\le &
	2\int_0^\pi \frac{\sin t}{t}  \, dt.	
\end{EQA}
To estimate \(I_{2}''\) we shall use the following inequalities 
\begin{EQA}
	&&|\sin(x+y) - \sin(x)| \le |y| \quad \text{ for all } x,y \in \R,\\
	&&0 \le \frac{\pi}{2} - \arcsin(1-z) \le 2^{\frac{3}{2}} z^{\frac{1}{2}} \quad \text{ for } 0 \le z \le 1.
\end{EQA}
Applying these inequalities we get that
\begin{EQA}
	\left|\sin\bigg(t b + \sum_{k=1}^{2} \frac{1}{2} \bigg( \phi_{k} - \frac \pi 2 \bigg) \bigg)- \sin(t b) \right| 
	&\le &
	\frac{c'}{\lambda_{2}^{2} t^{2}}, 
\end{EQA}
where \(c'\) is some absolute constant. Hence, 
\begin{EQA}
	I_3'' 
	&\le & 
	\frac{c'}{\lambda_{2}^{2}}\int_{1/\lambda_{2}}^{\infty} \frac{1}{t^3} \, dt \le c.
\end{EQA}
The estimates for \(I_3''\) and \(I_4''\) are similar. For simplicity we estimate \(I_3''\) only. Applying the following inequality
\begin{EQA}
	0 
	&\le &
	1 - \left( \frac{t^{2} \lambda_{k}^{2}}{1+4t^{2}\lambda^{2}_{k}}\right)^{1/4} \le \frac{1}{4 t^{2} \lambda_{2}^{2}}, \quad k = 1, 2,
\end{EQA}
we obtain that
\begin{EQA}
	I_3'' 
	&\le & 
	\frac{c''}{\lambda_{2}^{2}}\int_{1/\lambda_{2}}^{\infty} \frac{1}{t^3} \, dt \le c,
\end{EQA}
where \(c''\) is some absolute constant.
\end{proof}
	
\end{lemma}


\end{document}